\documentclass[onefignum,onetabnum]{siamart190516}

\def\ve{\varepsilon}

\def\R{\mathbb{R}}
\def\mc{\mathcal}
\def\mb{\mathbf}
\def\mbb{\mathbb}
\def\ra{\rightarrow}
\def\P{\mathbb{P}}

\def\W{\mathbf{W}}
\def\G{\mb{G}}
\def\J{\mb{J}}
\def\H{\mathbf{H}}
\def\I{\mb{I}}

\def\ve{\varepsilon}

\def\mbb{\mathbb}
\def\mb{\mathbf}
\def\mc{\mathcal}
\def\wh{\widehat}
\def\ul{\widetilde}
\def\ol{\overline}
\def\ul{\underline}

\def\bds{\boldsymbol}
\newcommand{\mn}[1]{{\left\vert\kern-0.25ex\left\vert\kern-0.25ex\left\vert\kern0.3ex #1 
    \kern0.3ex\right\vert\kern-0.25ex\right\vert\kern-0.25ex\right\vert}}

\usepackage{amsfonts}
\usepackage{graphicx}
\usepackage{epstopdf}
\usepackage[utf8]{inputenc}
\usepackage[T1]{fontenc}
\usepackage{mathtools}
\usepackage{makecell}
\usepackage{braket}

\newtheorem{assumption}{Assumption}[section]
\usepackage{mathtools}
\usepackage{xfrac}
\usepackage{microtype}
\usepackage{tabularx,amsmath,amssymb,graphicx,paralist,subfigure,pdfpages,cite,algorithm,paralist,multirow,comment}
\usepackage{appendix}
\usepackage{lipsum}
\usepackage{verbatim}
\usepackage{algorithmicx}
\usepackage{algpseudocode}
\usepackage{bbm}
\newsiamremark{remark}{Remark}

\def\ve{\varepsilon}

\def\mbb{\mathbb}
\def\mb{\mathbf}
\def\mc{\mathcal}
\def\wt{\widetilde}
\def\wh{\widehat}
\def\ol{\overline}
\def\ul{\underline}

\def\x{\mb{x}}
\def\y{\mb{y}}
\def\v{\mb{v}}
\def\a{\alpha}
\def\ra{\rightarrow}
\def\bds{\boldsymbol}

\def\H{\mc H}
\def\K{\mc K}
\def\b{\big}

\def\SGD{\texttt{SGD}}
\def\DSGD{{\texttt{DSGD}}}
\def\DSGT{{\texttt{DSGT}}}
\def\D2{{\texttt{D2}}}

\def\SARAH{{\texttt{SARAH}}}

\def\GS{{\texttt{GT-SARAH}}}

\def\gf{\nabla f}
\def\E{\mbb{E}}
\def\tsum{\textstyle{\sum}}
\def\ave{\frac{1}{n}(\mb{1}_n^\top\otimes \I_p)}
\def\F{\mc{F}^{t,s}}
\def\s{\sigma}
\def\t{\tau}
\def\n{\nonumber}
\def\mod{\mbox{mod}}
\def\e{\epsilon}
\DeclarePairedDelimiter\ceil{\lceil}{\rceil}
\DeclarePairedDelimiter\floor{\lfloor}{\rfloor}

\newenvironment{T1}
  {\begin{proof}[Proof~of~\cref{asp}]}
  {\end{proof}}

\newenvironment{T2}
  {\begin{proof}[Proof~of~\cref{out}]}
  {\end{proof}}
  
\newenvironment{T3}
  {\begin{proof}[Proof~of~\cref{main}]}
  {\end{proof}}

\newenvironment{L4}
  {\begin{proof}[Proof~of~Lemma~\ref{vrlem}]}
  {\end{proof}}

\begin{document}
\title{Fast decentralized non-convex finite-sum optimization with recursive variance reduction\thanks{This work has been partially supported by NSF under awards \#1513936, \#1903972, and \#1935555}.}
\author{
Ran Xin\thanks{Department of Electrical and Computer Engineering (ECE), Carnegie Mellon University, Pittsburgh, PA (\email{ranx@andrew.cmu.edu}, \email{soummyak@andrew.cmu.edu}).} 
\and
Usman A. Khan\thanks{Department of ECE, Tufts University, Medford, MA (\email{khan@ece.tufts.edu}).}
\and
Soummya Kar\footnotemark[2]}
\maketitle

\begin{abstract}
{\color{black}This paper considers decentralized minimization of~$N:=nm$ smooth non-convex cost functions equally divided over a directed network of~$n$ nodes. Specifically, we describe a stochastic first-order gradient method, called~\texttt{GT-SARAH}, that employs a \texttt{SARAH}-type variance reduction technique and gradient tracking (\texttt{GT}) to address the stochastic and decentralized nature of the problem. We show that \texttt{GT-SARAH}, with appropriate algorithmic parameters, finds an~$\epsilon$-accurate first-order stationary point with $\mc{O}\big(\max\big\{N^{\sfrac{1}{2}},n(1-\lambda)^{-2},n^{\sfrac{2}{3}}m^{\sfrac{1}{3}}(1-\lambda)^{-1}\big\}L\epsilon^{-2}\big)$ gradient complexity, where~${(1-\lambda)\in(0,1]}$ is the spectral gap of the network weight matrix and~$L$ is the smoothness parameter of the cost functions. This gradient complexity outperforms that of the existing decentralized stochastic gradient methods. In particular, in a big-data regime such that~${n = \mc{O}(N^{\sfrac{1}{2}}(1-\lambda)^{3})}$, this gradient complexity furthers reduces to~${\mc{O}(N^{\sfrac{1}{2}}L\epsilon^{-2})}$, independent of the network topology, and matches that of the centralized near-optimal variance-reduced methods. Moreover, in this regime \texttt{GT-SARAH} achieves a \emph{non-asymptotic linear speedup}, in that, the total number of gradient computations at each node is reduced by a factor of~$1/n$ compared to the centralized near-optimal algorithms that perform all gradient computations at a single node. To the best of our knowledge, \texttt{GT-SARAH} is the first algorithm that achieves this property. In addition, we show that appropriate choices of local minibatch size balance the trade-offs between the gradient and communication complexity of \texttt{GT-SARAH}.  Over infinite time horizon, we establish that all nodes in \texttt{GT-SARAH} asymptotically achieve consensus and converge to a first-order stationary point in the almost sure and mean-squared sense.}
\end{abstract}

\section{Introduction}
\label{sec:intro}
We consider decentralized finite-sum minimization of~${N:=nm}$ cost functions that takes the following form:
\begin{align}\label{P}
\min_{\mb{x}\in\mbb{R}^p}F(\mb x) := \tfrac{1}{n}\tsum_{i=1}^n f_i(\mb x), \qquad f_i(\mb x) := \tfrac{1}{m}\sum_{j=1}^{m}f_{i,j}(\mb x),
\end{align}
where each~$f_i:\mbb R^p\rightarrow\mbb R$, further decomposed as the average of~$m$ component costs $\{f_{i,j}\}_{j=1}^m$, is available only at the~$i$-th node in a network of~$n$ nodes. The network is abstracted as a directed graph ${\mc G := \{\mc{V},\mc{E}\}}$, where~${\mc{V} := \{1,\cdots,n\}}$ is the set of node indices and~$\mc{E}\subseteq\mc{V}\times\mc{V}$ is the collection of ordered pairs~${(i,r),i,r\in\mc{V}}$, such that node~$r$ sends information to node~$i$. {\color{black}We adopt the convention that~$(i,i)\in\mc{E},\forall i\in\mc{V}$.} Each node in the network is restricted to local computation and communication with its neighbors. Throughout the paper, we focus on the case where each~$f_{i,j}$ is differentiable, not necessarily convex, and~$F$ is bounded below. This formulation often appears in decentralized empirical risk minimization, where each local cost~$f_i$ can be considered as an empirical risk computed over a finite number of~$m$ local data samples~\cite{SPM_Xin}, and lies at the heart of many modern machine learning problems~\cite{OPT_ML,tutorial_nedich,DOPT_survey_Yang}. Examples include non-convex linear models and neural networks. When the local data size~$m$ is large, evaluating the exact gradient~$\nabla f_i$ of each local cost at each iteration becomes computationally expensive and methods that efficiently sample each local data batch are preferable. {\color{black}We are thus interested in designing fast stochastic gradient algorithms to find an~$\epsilon$-accurate first-order stationary point~$\wh{\x}\in\R^p$ such that~$\E\big[\|\nabla F(\wh{\x})\|^2\big]\leq\epsilon^2$}. 

Towards Problem~\eqref{P}, \DSGD~\cite{DGD_Yin,DSGD_nedich,diffusion_Chen,DSGD1_Chen}, a decentralized version of stochastic gradient descent (\SGD)~\cite{OPT_ML,SGD_Lan,SGD_Nemirovski}, is often used to address the large-scale and decentralized nature of the data.~\DSGD~is popular for several inference and learning tasks due to its simplicity of implementation and speedup in comparison to its centralized counterparts~\cite{DSGD_nips}. \texttt{DSGD} and its variants have been been extensively studied for different computation and communication needs, e.g., momentum~\cite{slowmo}, directed graphs~\cite{SGP_ICML}, escaping saddle-points~\cite{DSGD_Swenson,DSGD_vlaski_2}, zeroth-order schemes~\cite{DSGD_ZO_SICON_Hong}, swarming-based implementations~\cite{Swarming_Pu}, and constrained problems~\cite{DSGD_constraints_YDM}. 

\subsection{Challenges with DSGD} 
The performance of~\DSGD~for the non-convex Problem~\eqref{P} however suffers from three major challenges:
\begin{inparaenum}[(i)]
\item the non-degenerate variance of the stochastic gradients at each node; 
\item the dissimilarity among the local functions across the nodes; and
\item the transient time to reach the network topology independent region.
\end{inparaenum}
To elaborate these issues, we recap~\DSGD~for Problem~\eqref{P}~and its convergence results as follows. Let~$\mb x_i^k\in\mbb{R}^p$ denote the iterate of~\DSGD~at node~$i$ and iteration~$k$. At each node~$i$,~\DSGD~performs~\cite{DSGD_nedich,diffusion_Chen}
\begin{align}
\mb x_i^{k+1} = \tsum_{r=1}^n \ul w_{ir} \mb x_r^k - \alpha \cdot \mb{g}_i^k, \qquad k\geq0,
\end{align}
where~$\ul{\mb{W}}=\{\ul w_{ir}\}\in\mbb{R}^{n\times n}$ is a weight matrix that respects the network topology, while~$\mb{g}_i^k\in\mbb{R}^p$ is a stochastic gradient such that~$\E[\mb{g}_i^k|\mb{x}_i^k] = \nabla f_i(\mb{x}_i^k)$. 
Assuming the \emph{bounded variance} of each local stochastic gradient~$\mb{g}_i^k$, the \emph{bounded dissimilarity} between the local and the global gradient~\cite{DSGD_nips}, i.e., for some~$\nu>0$ and~$\zeta>0$,
\begin{align}\label{boundedness_asp}
\sup_{i\in\mc{V},k\geq0}\mathbb{E}\Big[\b\|\mb{g}_i^k-\nabla f_i(\mb x_i^k)\b\|^2\Big] \leq\nu^2 ~\mbox{and}~\sup_{\mb{x}\in\mbb{R}^p}\tfrac{1}{n}\tsum_{i=1}^{n}\left\|\nabla f_i(\mb{x}) - \nabla F(\mb{x})\right\|^2\leq\zeta^2,
\end{align}
and~$L$-smoothness of each~$f_i$, it is shown in~\cite{DSGD_nips} that, for small enough~$\a>0$,
\begin{equation}\label{DSGD_perf}
\frac{1}{K}\sum_{k=0}^{K-1}\E\Big[\big\|\nabla F(\ol{\mb{x}}^k)\big\|^2\Big] =
\mc{O}\left(\frac{F(\ol{\mb x}^0)-F^*}{\alpha K} + \frac{\alpha L\nu^2}{n}
+ \frac{\alpha^2L^2\nu^2}{1-\lambda} + \frac{\alpha^2L^2\zeta^2}{(1-\lambda)^2}\right),
\end{equation}
where~${\ol{\mb x}^k := \frac{1}{n}\sum_{i=1}^n\mb{x}_i^k}$ and~${(1-\lambda)}\in(0,1]$ is the spectral gap of the weight matrix~$\ul{\mb W}$. It then follows that~\cite{DSGD_nips} for~$K$ \emph{large enough}, see (iii) below, and with an appropriate step-size~$\alpha$, \textbf{\texttt{DSGD}} finds an~$\e$-accurate first-order stationary point of~$F$ in~$\mc O(\nu^2L\e^{-4})$ stochastic gradient computations across all nodes and therefore achieves \emph{asymptotic} linear speedup compared to the centralized~\SGD~\cite{OPT_ML,SGD_Lan} that executes at a single node. Clearly,~there are three issues with the convergence properties of~\DSGD:

(i)~Due to the non-degenerate stochastic gradient variance, the gradient complexity of \texttt{DSGD} does not match that of the centralized near-optimal variance-reduced methods when minimizing a finite-sum of smooth non-convex functions~\cite{sarah_ncvx,SPIDER,spiderboost}.

(ii)~The bounded dissimilarity assumption on the local and global gradients~\cite{DSGD_nips,DSGD_vlaski_2,SGP_ICML} or the coercivity of each local function~\cite{DSGD_Swenson} is essential for establishing the convergence of~\DSGD. 
In fact, a counterexample has been shown in~\cite{SPM_Hong} that \emph{\DSGD~diverges for any constant step-size} when these types of assumptions are violated.
Furthermore, the practical performance of~\DSGD~degrades significantly when the local and the global gradients are substantially different, i.e., when the data distributions across the nodes are largely heterogeneous~\cite{D2,SED,improved_DSGT_Xin}.  

(iii)~\DSGD~achieves~linear speedup only \emph{asymptotically}, i.e., after a finite number of transient iterations that is a polynomial function of~$n,\nu,\zeta,L,$ and~${(1-\lambda)}$~\cite{DSGD_nips,slowmo,DSGD_Pu}. 

\subsection{Main Contributions} 
This paper proposes~\texttt{GT-SARAH}, a novel decentralized stochastic variance-reduced gradient method that provably addresses the aforementioned challenges posed by~\DSGD. \texttt{GT-SARAH} is based on a \textit{local \texttt{SARAH}-type gradient estimator}~\cite{sarah_ncvx,SPIDER}, which removes the variance incurred by the local stochastic gradients, and \textit{global gradient tracking (\texttt{GT})}~\cite{NEXT_scutari,MP_scutari,AugDGM}, that fuses the gradient estimators across the nodes such that the bounded dissimilarity or the coercivity type assumptions are not required. Our main technical contributions are summarized in the following. 

{\color{black}
(i) {\color{black}We show that \texttt{GT-SARAH}, under appropriate algorithmic parameters, finds an~$\epsilon$-accurate first-order stationary point~$\wh{\x}$ of~$F$ such that~$\E\big[\|\nabla F(\wh{\x})\|^2\big]\leq\epsilon^2$ in} at most
$\mc{H}_R:=\mc{O}\big(\max\big\{N^{\sfrac{1}{2}},n(1-\lambda)^{-2},n^{\sfrac{2}{3}}m^{\sfrac{1}{3}}(1-\lambda)^{-1}\big\}L\epsilon^{-2}\big)$
component gradient computations across all nodes. {\color{black}The gradient complexity~$\H_R$ significantly outperforms that of the existing decentralized stochastic gradient algorithms for Problem~\eqref{P}; see~\cref{T1} for a formal comparison.} 

(ii) In a big-data regime such that~${n=\mc{O}(N^{\sfrac{1}{2}}(1-\lambda)^{3})}$, the gradient complexity $\mc{H}_R$ of \texttt{GT-SARAH} reduces to ${\wt{\mc{H}}_R := \mc{O}(N^{\sfrac{1}{2}}L\e^{-2})}$. {\color{black}We emphasize that~$\wt{\mc{H}}_R$ is independent of the network topology and matches that of the centralized near-optimal variance-reduced methods~\cite{sarah_ncvx,SPIDER,spiderboost}
under a slightly stronger smoothness assumption; see~\cref{D_C_smooth} for details.} Furthermore, since~\texttt{GT-SARAH} computes~$n$ gradients in parallel at each iteration, its per-node gradient complexity in this regime is~$\mc{O}(N^{\sfrac{1}{2}}n^{-1}\e^{-2})$, demonstrating a \emph{non-asymptotic linear speedup} compared with the aforementioned centralized near-optimal methods~\cite{SPIDER,spiderboost,sarah_ncvx} that perform all gradient computations at a single node. To the best of our knowledge, \texttt{GT-SARAH} is the first decentralized method that achieves this property for Problem~\eqref{P}.

(iii) {\color{black}We show that choosing the local minibatch size of \texttt{GT-SARAH} judiciously balances the trade-offs between the gradient and communication complexity; see~\cref{main1} and Subsection~\ref{two_regimes} for details.} 

(iv) We establish that all nodes in \texttt{GT-SARAH}
asymptotically achieve consensus and converge to a first-order stationary point of~$F$ over infinite time horizon in the almost sure and mean-squared sense.}

\subsection{Related work} Several algorithms have been proposed to improve certain aspects of \texttt{DSGD}. For example, a stochastic variant of \texttt{EXTRA}~\cite{EXTRA}, \texttt{Exact Diffusion}~\cite{SED}, and \texttt{NIDS}~\cite{NIDS}, called~\D2~\cite{D2}, removes the bounded dissimilarity assumption in \texttt{DSGD} based on a bias-correction principle. \texttt{DSGT} \cite{improved_DSGT_Xin}, introduced in~\cite{MP_Pu} for smooth and strongly convex problems, achieves a similar theoretical performance as~\D2 via gradient tracking~\cite{NEXT_scutari,harnessing,DIGing,GT_Wai}, but with more general choices of weight matrices. Reference~\cite{SPD_Lei} establishes asymptotic properties of a decentralized stochastic primal-dual algorithm for smooth convex problems. {\color{black}Reference~\cite{commslide_lan} develops decentralized primal-dual communication sliding algorithms that achieve communication efficiency for convex and possibly nonsmooth problems.}
These methods however are subject to the non-degenerate variance of the stochastic gradients. Inspired by the variance-reduction techniques for centralized stochastic optimization~\cite{Prox_SVRG,SARAH,SAGA,SVRG_allen,SVRG_reddi,spiderboost,SPIDER,sarah_ncvx,DVR_master_worker,inexact_SARAH,optimal_zhou}, decentralized variance-reduced methods for smooth and strongly-convex problems have been proposed recently, e.g., in~\cite{DSA,DAVRG,GTVR,SPM_Xin,Network-DANE}; in particular, the integration of gradient tracking and variance reduction described in this paper was introduced in~\cite{GTVR,SPM_Xin} to obtain linear convergence. 

{\color{black}
A recent paper~\cite{D_Get} proposes \texttt{D-GET} for Problem~\eqref{P}, which also considers local~\SARAH-type variance reduction and gradient tracking. In the following, we compare our work to~\cite{D_Get} from a few major technical aspects.\footnote{{\color{black}Note that~\cite{D_Get} uses $\E[\|\nabla F(\wh{\x})\|^2]\leq\epsilon$ as the performance metric, while we use $\E[\|\nabla F(\wh{\x})\|^2]\leq\epsilon^2$ in this paper. We state the complexities of \texttt{D-GET} established in~\cite{D_Get} under our metric for consistency.}}  
\emph{First}, the gradient complexity~$\H_R$ of \texttt{GT-SARAH} improves that of \texttt{D-GET} in terms of the dependence on~$n$ and~$m$; see~\cref{T1}. In particular, in a big-data regime,~$n = \mc{O}(N^{\sfrac{1}{2}}(1-\lambda)^3)$,~$\H_R$ matches the gradient complexity of the centralized near-optimal methods~\cite{sarah_ncvx,SPIDER,spiderboost}; in contrast, the gradient complexity of \texttt{D-GET} is worse than that of the centralized near-optimal methods by a factor of~$n^{\sfrac{1}{2}}$ even if the network is fully-connected. \emph{Second}, the complexity results of \texttt{D-GET} are attained with a specific local minibatch size~$m^{\sfrac{1}{2}}$. Conversely, we establish general complexity bounds of \texttt{GT-SARAH} with arbitrary local minibatch size and characterize the computation-communication trade-offs induced by different choices of the minibatch size. \emph{Third}, the Lyapunov function based convergence analysis of \texttt{D-GET} does not show explicit dependence of several important problem parameters, such as~$(1-\lambda)$ and~$L$, while the analysis in this work reveals explicitly the dependence of all problem related parameters and sheds light on their implications. \emph{Fourth}, we note that both \texttt{GT-SARAH} and \texttt{D-GET} achieve a worst case communication complexity of the form~$\mc{O}((1-\lambda)^{-a}L^b\epsilon^{-2})$, independent of~$m$ and~$n$, for some~$a,b\in\R^{+}$. Since the dependence of~$a$ and~$b$ in \texttt{D-GET} are not explicit, it is unclear which algorithm achieves a lower communication complexity. \emph{Finally}, \cite{D_Get} presents~a variant of \texttt{D-GET} that is applicable to a more general online setting such as expected risk minimization.}

\begin{table}[!ht]
\footnotesize
\renewcommand{\arraystretch}{2.3}
\setlength{\tabcolsep}{3.5pt}
\caption{A comparison of the gradient complexities of the-state-of-the-art decentralized stochastic gradient methods to minimize a sum of~${N=nm}$ smooth non-convex functions equally divided among~$n$ nodes. {\color{black}The gradient complexity is in terms of the total number of component gradient computations across all nodes to find a first-order stationary point~$\wh{\x}\in\R^p$ such that~$\E\big[\|\nabla F(\wh{\x})\|^2\big]\leq\epsilon^2$.} In the table,~$\nu^2$ denotes the bounded variance of the stochastic gradients described in~\eqref{boundedness_asp},~$(1-\lambda)\in(0,1]$ is the spectral gap of the network weight matrix and~$L$ is the smoothness parameter of the cost functions. We note that the complexities of \texttt{DSGD}, \texttt{D2}, \texttt{DSGT} in the table are established in the setting of stochastic first-order oracles, which is more general than the finite-sum formulation considered in this paper. Moreover, the complexities of \texttt{DSGD}, \texttt{D2}, \texttt{DSGT} in the table are stated in the regime that~$\epsilon$ is small enough for simplicity; see~\cite{DSGD_nips,D2,improved_DSGT_Xin} for their precise expressions. Finally, we note that only the best possible gradient complexity of \texttt{GT-SARAH}, in the sense of \cref{main}, is presented in the table for conciseness; see~\cref{main1} and Subsection~\ref{two_regimes} for detailed discussion on balancing the trade-offs between the gradient and communication complexity of \texttt{GT-SARAH}.}
\label{T1}
\begin{center}
\begin{tabular}{|c|c|c|}
\hline
\textbf{Algorithm} & \textbf{Gradient complexity} & \textbf{Remarks}\\ \hline
\DSGD~\cite{DSGD_nips} & $\mc{O}\left(\dfrac{\nu^2L}{\epsilon^{4}}\right)$ & \makecell{bounded variance, \\ bounded dissimilarity}  \\ \hline
\D2~\cite{D2} & $\mc{O}\left(\dfrac{\nu^2L}{\epsilon^{4}}\right)$ & bounded variance \\ \hline
\DSGT~\cite{improved_DSGT_Xin} & $\mc{O}\left(\dfrac{\nu^2L}{\epsilon^{4}}\right)$ & bounded variance \\ \hline
\texttt{D-GET}~\cite{D_Get} &  $\mc{O}\bigg(\dfrac{n^{\sfrac{1}{2}}N^{\sfrac{1}{2}}L^b}{(1-\lambda)^{a}\e^2}\bigg)$ & \makecell{$a,b\in\R^{+}$
are not \\ explicitly shown in~\cite{D_Get}}  \\ \hline 
\makecell{\GS\\(\textbf{this work})}  &  $\mc{O}\bigg(\max\bigg\{N^{\sfrac{1}{2}},\dfrac{n}{(1-\lambda)^2},\dfrac{n^{\sfrac{2}{3}}m^{\sfrac{1}{3}}}{1-\lambda}\bigg\}\dfrac{L}{\epsilon^2}\bigg)$ & \makecell{See~\cref{main} and \\ \cref{main1}} \\  
\hline
\end{tabular}
\end{center}
\end{table}

\subsection{Paper outline and notation}
The proposed~\GS~algorithm is developed in Section~\ref{alg_deve}. We present the convergence results of~\GS~and discuss their implications in Section~\ref{main_results}. Section~\ref{conv_analysis} presents the convergence analysis. Section~\ref{numer} presents numerical experiments while Section~\ref{concl} concludes the paper.

The set of positive integers and real numbers are denoted by~$\mbb{Z}^{+}$ and~$\mbb{R}^{+}$ respectively. For any $a\in\mbb{R}$,~$\floor*{a}$ denotes the largest integer~$i$ such that~$i\leq a$; similarly, $\ceil*{a}$ denotes the smallest integer~$i$ such that~$i\geq a$; 
We use lowercase bold letters to denote column vectors and uppercase bold letters to denote matrices. The matrix,~$\mb{I}_d$, represents the~$d\times d$ identity; $\mb{1}_d$ and $\mb{0}_d$ are the~$d$-dimensional column vectors of all ones and zeros, respectively. The Kronecker product of two matrices~$\mb{A}$ and~$\mb{B}$ is denoted by~$\mb{A}\otimes \mb{B}$. We use~$\|\cdot\|$ to denote the Euclidean norm of a vector or the spectral norm of a matrix. For a matrix~$\mb{X}$, we use~$\rho(\mb{X})$ to denote its spectral radius,~$\lambda_2(\mb{X})$ to denote its second largest singular value, and~$\det(\mb{X})$ to denote its determinant. Matrix inequalities are interpreted in the entry-wise sense. We use~$\sigma(\cdot)$ to denote the~$\s$-algebra generated by the random variables and/or sets in its argument. The empty set is denoted by~$\phi$.

\section{Algorithm Development:~\GS}\label{alg_deve}
We now systematically build the proposed algorithm~\GS~and provide the basic intuition. We recall that the performance~\eqref{DSGD_perf} of~\DSGD, in addition to the first term which is similar to that of the centralized batch gradient descent, has three additional bias terms. The second and third bias terms in~\eqref{DSGD_perf} depend on the variance~$\nu^2$ of local stochastic gradients. A variance-reduced gradient estimation procedure of~\SARAH-type~\cite{SPIDER,sarah_ncvx}, employed locally at each node~$i$ in \texttt{GT-SARAH}, removes~$\nu^2$. The last bias term in~\eqref{DSGD_perf} is due to the dissimilarity~$\zeta^2$ between the local gradients~$\{\nabla f_i\}_{i=1}^n$ and the global gradient~$\nabla F$. A dynamic fusion mechanism, called gradient tracking~\cite{harnessing,NEXT_scutari,DIGing,AugDGM,GT_jakovetic}, removes~$\zeta^2$ by tracking the average of the local gradient estimators in \texttt{GT-SARAH} to learn the global gradient at each node. The resulting algorithm is illustrated in Fig.~\ref{GTVRflow}.
\begin{figure}[!ht]
\centering
\includegraphics[width=5.1in]{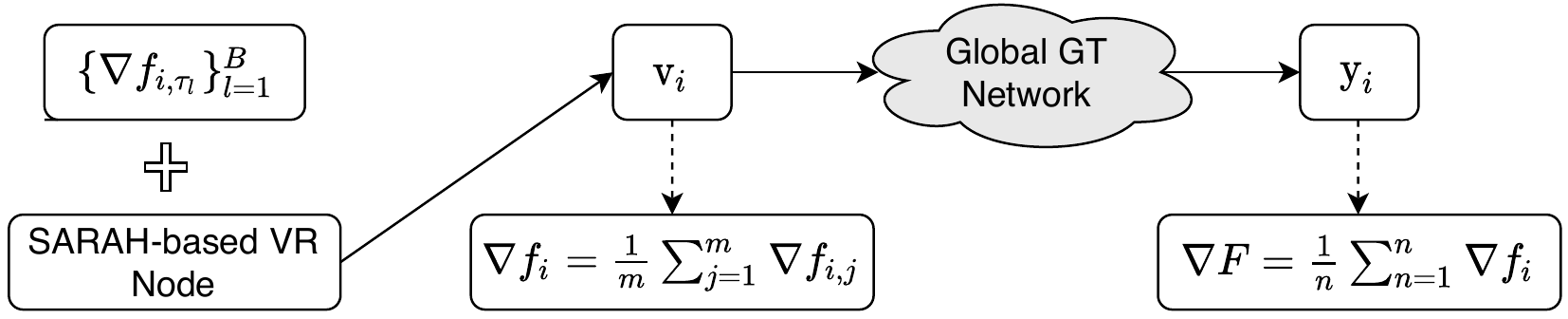}
\vspace{-0.4cm}
\caption{Each node~$i$ samples a minibatch of stochastic gradients~$\{\nabla f_{i,\tau_{l}}\}_{l=1}^B$ at each iteration from its local data batch and computes an estimator~$\mb v_i$ of its local batch gradient~$\nabla f_i$ via a~\SARAH-type~variance reduction (VR) procedure. These local gradient estimators~$\mb v_i$'s are then fused over the network via a gradient tracking technique to obtain~$\mb y_i$'s that approximate the global gradient~$\nabla F$.}
\label{GTVRflow}
\end{figure}

\subsection{Detailed Implementation} The complete implementation of~\GS~is summarized in~\cref{GT-SARAH}, where we assume that all nodes start from the same point~$\ol{\mb{x}}^{0,1}\in\mbb{R}^p$.
\GS~can be interpreted as a double loop method with an outer loop, indexed by~$s$, and an inner loop, indexed by~$t$. At the beginning of each outer loop~$s$,~\GS~computes the local batch gradient~${\mb v_i^{0,s}:= \nabla f_i(\mb{x}_i^{0,s})}$ at each node~$i$. These batch gradients are then used to compute the first iteration of the global gradient tracker~$\mb y_{i}^{1,s}$ and the state update~$\mb x_i^{1,s}$. The three quantities,~$\mb v_i^{0,s},\mb y_{i}^{1,s},\mb x_i^{1,s}$, set up the subsequent inner loop iterations. At each inner loop iteration~$t\geq 1$, each node~$i$ samples two minibatch stochastic gradients
from its local data that are used to construct the gradient estimator~$\mb v_i^{t,s}$. We note that the gradient estimator is of recursive nature, i.e., it depends on~$\mb v_i^{t-1,s}$ and the minibatch stochastic gradients evaluated at the current and the past states~$\mb x_i^{t,s}$ and~$\mb x_i^{t-1,s}$. The next step is to update~$\mb y_i^{t+1,s}$ based on the gradient tracking protocol. Finally, the state~$\mb x_i^{t+1,s}$ at each node~$i$ is computed as a convex combination of the states of the neighboring nodes followed by a descent in the direction of the gradient tracker~$\mb y_i^{t+1,s}$. The latest updates~$\mb{x}_i^{q+1,s}$,~$\mb{y}_i^{q+1,s}$ and~$\mb{v}_i^{q,s}$ then set up the next inner-outer loop cycle of~\GS.  

\setlength{\textfloatsep}{10pt}
\begin{algorithm}
\caption{\GS~at each node~$i$}
\label{GT-SARAH}
\begin{algorithmic}[1]
\Require{${\x_i^{0,1} \!= \ol{\mb{x}}^{0,1}\in\mbb R^p}$, ${\alpha\in\R^{+}}$, $q\in\mbb Z^+$,~$S\in\mbb Z^+$,~$B\in\mbb Z^+$,~$\{\ul{w}_{ir}\}_{r=1}^n$,
${\mb{y}_i^{0,1} = \mb{0}_p}$, ${\mb{v}_i^{-1,1} = \mb{0}_p}$}.
\For{$s = 1,2,\cdots,S$}
\vspace{0.1cm}
\State{$\v_i^{0,s} = \gf_i(\x_i^{0,s}) = \frac{1}{m}\sum_{j=1}^m\nabla f_{i,j}(\x_i^{0,s})$;} \Comment{batch gradient computation}
\vspace{0.1cm}
\State{$\mb{y}_{i}^{1,s} = \sum_{r=1}^{n}\ul{w}_{ir}\mb{y}_i^{0,s} + \mb{v}_i^{0,s} - \mb{v}_i^{-1,s}$} \Comment{gradient tracking}
\vspace{0.1cm}
\State{$\mb{x}_{i}^{1,s} = \sum_{r=1}^{n}\ul{w}_{ir}\mb{x}_{r}^{0,s} - \alpha\mb{y}_{i}^{1,s}$} \Comment{state update}
\vspace{0.1cm}
\For{$t = 1,2,\cdots,q$}
\vspace{0.1cm}
\State{{\color{black}for~$l$ in~$\{1,\cdots,B\}$, choose~$\tau_{i,l}^{t,s}$ uniformly at random from~$\{1,\cdots,m\}$};\\
\Comment{sampling}}
\vspace{0.1cm}
\State{{\color{black}$\mb{v}_{i}^{t,s} = \frac{1}{B}\sum_{l=1}^B\Big(\gf_{i,\tau_{i,l}^{t,s}}\big(\x_{i}^{t,s}\big) - \gf_{i,\tau_{i,l}^{t,s}}\big(\x_{i}^{t-1,s}\big)\Big) + \mb{v}_{i}^{t-1,s};$}}\Comment{SARAH}
\vspace{0.1cm}
\State{$\mb{y}_{i}^{t+1,s} = \sum_{r=1}^{n}\ul{w}_{ir}\mb{y}_{r}^{t,s} + \mb{v}_i^{t,s} - \mb{v}_i^{t-1,s};$}
\Comment{gradient tracking}
\vspace{0.1cm}
\State{$\mb{x}_{i}^{t+1,s} = \sum_{r=1}^{n}\ul{w}_{ir}\mb{x}_{r}^{t,s} - \alpha\mb{y}_{i}^{t+1,s};$}
\Comment{state update}
\vspace{0.1cm}
\EndFor
\vspace{0.2cm}
\State{Set~$\x_i^{0,s+1} = \x_i^{q+1,s}$;~$\y_i^{0,s+1} = \y_i^{q+1,s}$;~$\v_i^{-1,s+1}=\v_i^{q,s}$.} 
\Comment{next cycle}
\vspace{0.2cm}
\EndFor
\end{algorithmic}
\end{algorithm}

\section{Main Results}\label{main_results}
In this section, we present the main convergence results of \texttt{GT-SARAH} and discuss their implications.

\subsection{Assumptions}
We make the following assumptions to establish the convergence properties of~\GS~in this paper.
\begin{assumption}\label{smooth}
Each local component cost~$f_{i,j}$ is differentiable and~$\{f_{i,j}\}_{j=1}^m$ satisfies a mean-squared smoothness property, i.e., for some~$L>0$,
\begin{align}\label{d_smooth}
{\color{black}\tfrac{1}{m}\tsum_{j=1}^m\|\nabla f_{i,j}(\mb{x})-\nabla f_{i,j}(\mb{y})\|^2\leq L^2\|\mb x - \mb y\|^2,\quad\forall i\in\mc{V},\quad\forall\mb{x},\mb{y}\in\mbb{R}^p.}     
\end{align}
In addition, the global cost~$F$ is bounded below, i.e.,~$F^* := \inf_{\x\in\R^p}F(\x) > -\infty$.
\end{assumption}
It is clear that under Assumption~\ref{smooth}, each~$f_i$ and~$F$ are~$L$-smooth. We note that Assumption~\ref{smooth} is weaker than requiring each~$f_{i,j}$ to be~$L$-smooth.

\begin{remark}\label{D_C_smooth}
{\color{black}The local mean-squared smoothness assumption~\eqref{d_smooth}, which is also used in the existing work~\cite{D_Get}, is slightly stronger than the smoothness assumption required by the existing lower bound $\Omega(N^{\sfrac{1}{2}}L\epsilon^{-2})$~\cite{SPIDER,lower_zhou} and the centralized near-optimal methods~\cite{sarah_ncvx,SPIDER,spiderboost} for finite-sum problems in the following sense. If we view Problem~\eqref{P} as a centralized optimization problem, that is, all~$f_{i,j}$'s are available at a single node, then the aforementioned lower bound and the convergence of the centralized near-optimal methods are established under the following assumption: 
\begin{align}\label{c_smooth}
\tfrac{1}{nm}\tsum_{i=1}^n\tsum_{j=1}^m\|\nabla f_{i,j}(\mb{x})-\nabla f_{i,j}(\mb{y})\|^2\leq L^2\|\mb x - \mb y\|^2,  \quad\forall\mb{x},\mb{y}\in\mbb{R}^p.
\end{align}
Clearly,~\eqref{c_smooth} is implied by~\eqref{d_smooth} but not vice versa. Due to this subtle difference, it is unclear whether the existing lower bound~$\Omega(N^{\sfrac{1}{2}}L\epsilon^{-2})$~\cite{SPIDER,lower_zhou} established under~\eqref{c_smooth} remains valid under~\eqref{d_smooth}. Finally, we note that a lower bound result for decentralized deterministic first-order algorithms in the case of~$m = 1$ can be found in~\cite{lower_sun}.}
\end{remark}

\begin{assumption}\label{sampling}
The family~$\big\{\tau_{i,l}^{t,s}: t\in[1,q], s\geq1, i\in\mc{V},l\in[1,B]\big\}$ of random variables is independent.
\end{assumption}
Assumption~\ref{sampling} is standard in the stochastic optimization literature, e.g.,~\cite{SPIDER,OPT_ML}.

{\color{black}\begin{assumption}\label{network}
The nonnegative weight matrix~$\ul{\mb W} := \{\ul{w}_{ir}\}\in\mathbb{R}^{n \times n}$ associated with the network~$\mc{G} = (\mc{V},\mc{E})$ has positive diagonals and is primitive. Moreover,~$\ul{\mb W}$ is doubly stochastic, i.e., $\ul{\mb W}\mb{1}_n = \mb{1}_n$ and $\mb{1}_n^\top \ul{\mb W} = \mb{1}_n^\top.$
\end{assumption}
An important consequence of Assumption~\ref{network}
is that~\cite{harnessing} 
\begin{align}\label{spectral}
\lambda := \left\|\ul{\mb{W}}-\tfrac{1}{n}\mb{1}_n\mb{1}_n^\top\right\| = \lambda_2(\ul{\mb W}) \in [0,1),
\end{align}
where~$\lambda_2(\ul{\mb W})$ denotes the second largest singular value of~$\ul{\mb W}$.\footnote{We note that the relation in~\eqref{spectral} may be established by following the definition of the spectral norm with the help of the primitivity and doubly stochasticity of~$\ul{\mb{W}}$ and~$\ul{\mb{W}}^{\top}\ul{\mb{W}}$, Perron-Frobenius theorem, and the spectral decomposition of~$\ul{\mb{W}}^{\top}\ul{\mb{W}}$~\cite{matrix_analysis,harnessing}.} We term~$(1-\lambda)$ as the spectral gap of~$\ul{\mb W}$ that characterizes the connectivity of the network~\cite{tutorial_nedich}. }

{\color{black}
\begin{remark}
Weight matrices satisfying Assumption~\ref{network} may be designed for the family of strongly-connected directed graphs that admit doubly-stochastic weights: (i) towards the primitivity requirement in Assumption~\ref{network}, we note that if a graph is strongly-connected, then its associated weight matrix~$\mb{\ul{W}}$ is irreducible~\cite[Theorem 6.2.14, 6.2.24]{matrix_analysis} and~$\mb{\ul{W}}$ is further primitive since it is nonnegative with positive diagonals~\cite[Lemma 8.5.4]{matrix_analysis}; (ii) towards the doubly stochastic requirement in Assumption~\ref{network}, we refer the readers to~\cite{DS_digraph} for necessary and sufficient conditions under which a strongly connected directed graph admits doubly stochastic weights. 

An important special case of this family is undirected connected graphs where doubly stochastic weights always exist and can be constructed in an efficient and decentralized manner, for instance, by the lazy Metroplis rule~\cite{tutorial_nedich}. Hence, Assumption~\ref{network} is \emph{more general} than the one required by \texttt{EXTRA}-based algorithms for decentralized optimization. For example, the weight matrix of~\D2 needs to be symmetric and meet certain spectral properties~\cite{D2} and is therefore not applicable to directed graphs.
\end{remark}}

In the rest of the paper, we fix a rich enough probability space~$\left(\Xi,\mc{F},\mbb{P}\right)$ where all random variables generated by \texttt{GT-SARAH} are properly defined. We formally state the convergence results of~\GS~next, the proofs of which are deferred to Subsection~\ref{Proof_main}.

\subsection{Asymptotic almost sure and mean-squared convergence} The following theorem shows the asymptotic convergence of \texttt{GT-SARAH}.
\begin{theorem} \label{asp}
Let Assumptions~\ref{smooth}-\ref{network} hold. Suppose that the step-size~$\a$, minibatch size~$B$, and the inner-loop length~$q$ of \textbf{\texttt{GT-SARAH}} follow~$$0<\alpha\leq  \min\bigg\{\frac{(1-\lambda^2)^2}{4\sqrt{42}},~\bigg(\frac{nB}{6q}\bigg)^{\sfrac{1}{2}},~\bigg(\frac{4nB}{7nB+24q}\bigg)^{\sfrac{1}{4}}\frac{1-\lambda^2}{6}\bigg\}\frac{1}{2L},$$ where~$B\in[1,m]$. Then we have:~$\forall t\in[0,q]$,~$\forall i\in\mc{V}$, 
\begin{align*}
&\mbb{P}\left(\lim_{s\rightarrow\infty}\big\|\nabla F(\x_i^{t,s})\big\| = 0\right) = 1 \qquad \mbox{and} \qquad
\lim_{s\rightarrow\infty}\E\Big[\big\|\nabla F(\x_i^{t,s})\big\|^2\Big] = 0,         \\
&{\color{black}\mbb{P}\left(\lim_{s\rightarrow\infty}\big\|\mb{x}_i^{t,s}-\ol{\mb{x}}^{t,s}\big\| = 0\right) = 1 \quad~ \mbox{and} \qquad
\lim_{s\rightarrow\infty}\E\Big[\big\|\mb{x}_i^{t,s}-\ol{\mb{x}}^{t,s}\big\|^2\Big] = 0,}
\end{align*}
where~$\ol{\mb{x}}^{t,s} := \frac{1}{n}\sum_{i=1}^n\mb{x}_i^{t,s}$.
\end{theorem}
In addition to the mean-squared convergence that is standard in the stochastic optimization literature, the almost sure convergence in Theorem~\ref{asp} guarantees that all nodes in \texttt{GT-SARAH} asymptotically achieve consensus and converge to a first-order stationary point of~$F$ on almost every sample path. 

\subsection{Complexities of~\GS~for finding first-order stationary points}
We measure the outer-loop complexity of \texttt{GT-SARAH} in the following sense.
\begin{definition}\label{FSL}
Consider~the sequence of random state vectors~$\{\mb{x}_i^{t,s}\}$ generated by \texttt{GT-SARAH}, at each node~$i$. We say that \texttt{GT-SARAH} finds an~$\epsilon$-accurate first-order stationary point of~$F$ in~$S$ outer-loop iterations if
\begin{align}\label{FSL_ieq}
{\color{black}\frac{1}{S(q+1)}\sum_{s=1}^{S}\sum_{t=0}^{q}\frac{1}{n}\sum_{i=1}^n\E\Big[\big\|\nabla F\big(\x_i^{t,s}\big)\big\|^2
+ L^2\big\|\mb{x}_i^{t,s}-\ol{\mb{x}}^{t,s}\big\|^2\Big] \leq \e^2.}
\end{align}
\end{definition}
This is a standard metric that is concerned with the minimum of the stationary gaps and consensus errors over iterations in the mean-squared sense at each node~\cite{DSGD_nips,D2,sarah_ncvx,SPIDER,spiderboost}. {\color{black}In particular, if~\eqref{FSL_ieq} holds and the output~$\wh{\x}$ of~\GS~is chosen uniformly at random from the set~$\{\x_i^{t,s}: 0\leq t\leq q, 1\leq s\leq S, i\in\mc{V}\}$, then we have~$\E[\|\nabla F(\wh{\x})\|^2]\leq\epsilon^2$.} In the following, we first provide the outer-loop iteration complexity of~\GS.

\begin{theorem}\label{out}
Let Assumptions~\ref{smooth}-\ref{network} hold. Suppose that the step-size~$\a$, minibatch size~$B$, and the inner-loop length~$q$ of \textbf{\texttt{GT-SARAH}} follow~$$0<\alpha\leq\min\bigg\{\frac{(1-\lambda^2)^2}{4\sqrt{42}},~\left(\frac{nB}{6q}\right)^{\sfrac{1}{2}},~\left(\frac{4nB}{7nB+24q}\right)^{\sfrac{1}{3}}\frac{1-\lambda^2}{6}\bigg\}\frac{1}{2L},$$ where~$B\in[1,m]$. Then the number of the outer-loop iterations~$S$ required by \texttt{GT-SARAH} to find an~$\e$-accurate stationary point of~$F$ is at most 
$$\frac{1}{(q+1)\a L\epsilon^2}\Bigg(4L\left(F(\ol{\x}^{0,1}) - F^*\right) 
+ \frac{1}{n}\sum_{i=1}^n\left\|\nabla f_i(\ol{\x}^{0,1})\right\|^2\Bigg).$$
\end{theorem}
With~\cref{out} at hand, the gradient and communication complexities of \texttt{GT-SARAH} can be readily established. 

{\color{black}\begin{theorem}\label{main}
Let Assumptions~\ref{smooth}-\ref{network} hold.
Suppose that the step-size~$\a$ and the length~$q$ of the inner loop of \texttt{GT-SARAH} are chosen as\footnote{The~$\mc{O}$ notation only hides universal constants that are independent of problem parameters.} 
\begin{align}\label{alpha_q}
q = \mc{O}\Big(\frac{m}{B}\Big)
\quad
\mbox{and} 
\quad
\a = \mc{O}\bigg(\min\bigg\{(1-\lambda)^2,~\frac{n^{\sfrac{1}{2}}B}{m^{\sfrac{1}{2}}},~\frac{n^{\sfrac{1}{3}}B^{\sfrac{2}{3}}(1-\lambda)}{m^{\sfrac{1}{3}}}\bigg\}~\dfrac{1}{L}\bigg),
\end{align}
where~$B\in[1,m]$. Then \texttt{GT-SARAH} finds an~$\e$-accurate stationary point of~$F$ in
\begin{align*}
\mc{H}_B := \mc{O}\bigg(\max\bigg\{\frac{nB}{(1-\lambda)^{2}},~N^{\sfrac{1}{2}},&~\frac{m^{\sfrac{1}{3}}n^{\sfrac{2}{3}}B^{\sfrac{1}{3}}}{1-\lambda}\bigg\}\dfrac{\Delta}{\e^2}\bigg)
\end{align*}
component gradient computations across all nodes and
\begin{align*}
\mc{K}_B := \mc{O}\bigg(\max\bigg\{\frac{1}{(1-\lambda)^{2}},~\frac{m^{\sfrac{1}{2}}}{n^{\sfrac{1}{2}}B},~\frac{m^{\sfrac{1}{3}}}{n^{\sfrac{1}{3}}B^{\sfrac{2}{3}}(1-\lambda)}\bigg\}\dfrac{\Delta}{\e^2}\bigg)
\end{align*}
rounds of communication, where~$\Delta := L\left(F\big(\ol{\x}^{0,1}\big) - F^*\right) 
+ \frac{1}{n}\sum_{i=1}^n\|\nabla f_i(\ol{\x}^{0,1})\|^2$.
\end{theorem}} 

\begin{remark}
\cref{main} holds for an arbitrary minibatch size~$B\in[1,m]$.
\end{remark}
\begin{remark}
The gradient complexity \emph{at each node} of \GS~is~$\mc{H}_{B}/n$.
\end{remark}

In view of~\cref{main}, as the minibatch size~$B$ increases, the gradient complexity $\H_B$ (resp. the communication complexity $\K_B$) of \texttt{GT-SARAH} is non-decreasing (resp. non-increasing). The following corollary may be obtained from~\cref{main} by standard algebraic manipulations and shows that choosing the minibatch size~$B$ appropriately leads to favorable computation and communication trade-offs.

{\color{black}
\begin{corollary}\label{main1}
Let Assumptions~\ref{smooth}-\ref{network} hold.
Suppose that the step-size~$\a$ and the inner-loop length~$q$ of \texttt{GT-SARAH} are chosen according to~\eqref{alpha_q}. We have the following complexity results.

$(i)$ If $B\in[1, \floor{R}]$, where~$R:= \max\b\{m^{\sfrac{1}{2}}n^{-\sfrac{1}{2}}(1-\lambda)^3,1\b\}$, then \texttt{GT-SARAH} attains the best possible, in the sense of~\cref{main}, gradient complexity
\begin{align}\label{H_g}
\H_{R} := \mc{O}\bigg(\max\bigg\{\frac{n}{(1-\lambda)^{2}},~N^{\sfrac{1}{2}},&~\frac{m^{\sfrac{1}{3}}n^{\sfrac{2}{3}}}{1-\lambda}\bigg\}\dfrac{\Delta}{\e^2}\bigg);    
\end{align}
moreover, when~$B = \floor{R}$, the corresponding communication complexity of \texttt{GT-SARAH} is 
\begin{align}\label{K_g}
\K_{R} := 
\mc{O}\bigg(\!\max\bigg\{\frac{1}{(1-\lambda)^{2}}, \min\bigg\{\frac{m^{\sfrac{1}{2}}}{n^{\sfrac{1}{2}}},\frac{1}{(1-\lambda)^{3}}\bigg\}, \min\bigg\{\frac{m^{\sfrac{1}{3}}}{n^{\sfrac{1}{3}}(1-\lambda)},\frac{1}{(1-\lambda)^{3}}\bigg\}\bigg\}\dfrac{\Delta}{\e^2}\bigg).
\end{align}

$(ii)$ If $B\in[\ceil{C}, m]$, where~$C := \max\b\{m^{\sfrac{1}{2}}n^{-\sfrac{1}{2}}(1-\lambda)^{\sfrac{3}{2}},1\b\}$, then \texttt{GT-SARAH}~attains the best possible, in the sense of~\cref{main}, communication complexity
\begin{align}\label{K_c}
\K_{C} := \mc{O}\bigg(\frac{1}{(1-\lambda)^2}\frac{\Delta}{\epsilon^2}\bigg);    
\end{align}
moreover, when~$B = \ceil{C}$, the corresponding gradient complexity of~\GS~is
\begin{align}\label{H_c}
\H_{C} := \mc{O}\bigg(\max\bigg\{\frac{n}{(1-\lambda)^{2}},~\frac{N^{\sfrac{1}{2}}}{(1-\lambda)^{\sfrac{1}{2}}},&~\frac{m^{\sfrac{1}{3}}n^{\sfrac{2}{3}}}{1-\lambda}\bigg\}\dfrac{\Delta}{\e^2}\bigg).    
\end{align}
\end{corollary}

Comparing~\eqref{H_g}~\eqref{K_g} with~\eqref{H_c}~\eqref{K_c}, we clearly have~$\H_{R}\leq\H_{C}$ and~$\K_{R}\geq\K_{C}$.

\subsubsection{Two regimes of practical significance}\label{two_regimes}
{\color{black}We now discuss the implications of the complexity results in~\cref{main1} and the corresponding computation-communication trade-offs in the following regimes of practical significance.

\textbf{$\bullet$ Big-data regime:}~${n = \mc{O}(N^{\sfrac{1}{2}}(1-\lambda)^{3})}$. In this regime, typical to large-scale machine learning, i.e., the total number of data samples~$N$ is very large, it can be verified that $\H_{R}$ reduces to $\wt{\H}_{R}: = \mc O(N^{\sfrac{1}{2}}\Delta\e^{-2})$ and $\K_{R}$ reduces to~$\wt{\K}_{R} := \mc{O}((1-\lambda)^{-3}\Delta\e^{-2})$.
It is worth noting that $\wt{\H}_{R}$ is independent of the network topology and matches the gradient complexity of the centralized near-optimal variance-reduced methods \cite{sarah_ncvx,spiderboost,SPIDER} for this problem class up to constant factors, under a slightly stronger smoothness assumption; see~\cref{D_C_smooth}. Moreover,~$\wt{\H}_{R}$ demonstrates a non-asymptotic linear speedup in that the number of component gradient computations required \textit{at each node} to achieve an~$\epsilon$-accurate stationary point of~$F$ is reduced by a factor of~$1/n$, compared to the aforementioned centralized near-optimal algorithms \cite{SPIDER,spiderboost,sarah_ncvx} that perform all gradient computations at a single node. 

On the other hand, it is straightforward to verify that~$\H_{C}$ reduces to~$\wt{\H}_{C} := \mc
O(N^{\sfrac{1}{2}}(1-\lambda)^{-\sfrac{1}{2}}\Delta\e^{-2})$. In other words, in this big-data regime, choosing a large minibatch size~$B = \ceil{C}$ improves the communication complexity from~$\wt{\K}_{R}$ to~$\K_{C}$ while deteriorates the gradient complexity from~$\wt{\H}_{R}$ to~$\wt{\H}_{C}$, demonstrating an interesting trade-off between computation and communication.

\textbf{$\bullet$ Large-scale network regime:}~${n=\Omega(N^{\sfrac{1}{2}}(1-\lambda)^{\sfrac{3}{2}})}$. In this regime, typical to ad hoc IoT networks, i.e., the number of the nodes~$n$ and the network spectral gap inverse~$(1-\lambda)^{-1}$ are large compared with the total number of samples~$N$, it can be verified that~$R = C = 1$ and consequently~$\H_{R} = \H_{C}$ reduce to $\mc{O}(n(1-\lambda)^{-2}\Delta\e^{-2})$ while~$\K_{R} = \K_{C}$ reduce to $\mc{O}((1-\lambda)^{-2}\Delta\e^{-2})$. In other words, in this large-scale network regime, the minibatch size~$B = \mc{O}(1)$ is preferred since it attains the best possible gradient and communication complexity simultaneously, in the sense of \cref{main}. }}

\begin{remark}[Characterization of the big-data regime] We note that the number of nodes~$n$ may be interpreted as the intrinsic minibatch size of \texttt{GT-SARAH}. We recall that the centralized near-optimal variance-reduced algorithms~\cite{SPIDER,spiderboost,sarah_ncvx} for this problem class retain their best possible gradient complexity if their minibatch size does not exceed~$N^{\sfrac{1}{2}}$~\cite{sarah_ncvx}. Thus, the aforementioned big-data regime~$n =  \mc{O}(N^{\sfrac{1}{2}}(1-\lambda)^3)$ approaches the centralized one as the network connectivity improves and matches the centralized one when the network is fully connected, i.e.,~$\lambda=0$.
\end{remark}

\section{Convergence Analysis}\label{conv_analysis}
In this section, we present the proof pipeline for Theorems~\ref{asp},~\ref{out}, and~\ref{main}. The analysis framework is novel and general and may be applied to other decentralized algorithms built around variance reduction and gradient tracking.
To proceed, we first write~\GS~in a matrix form. Recall that~\GS~is a double loop method, where the outer loop index is~$s\in\{1,\ldots,S\}$ and the inner loop index is~$t\in\{0,\ldots,q\}$. It is straightforward to verify that~\GS~can be equivalently written as:~$\forall s\geq1$ and~$t\in[0,q]$,
\begin{subequations}
\begin{align}
\mb{y}^{t+1,s} &= \mb{W}\mb{y}^{t,s} + \mb{v}^{t,s} - \mb{v}^{t-1,s}, \label{gtsarah1} \\
\mb{x}^{t+1,s} &= \mb{W}\mb{x}^{t,s} - \alpha\mb{y}^{t+1,s},  \label{gtsarah2} 
\end{align}
\end{subequations}
where~$\mb v^{t,s}, \mb{x}^{t,s},$ and~$\mb y^{t,s}$, in~$\mathbb{R}^{np}$, that concatenate local gradient estimators~${\{\mb{v}_i^{t,s}\}_{i=1}^n}$, states~${\{\mb{x}_i^{t,s}\}_{i=1}^n}$, and gradient trackers~${\{\mb{y}_i^{t,s}\}_{i=1}^n}$, respectively, and~${\mb{W} := \ul{\mb{W}}\otimes \I_p}$. 
{\color{black}We recall that~$\x^{0,s+1} = \x^{q+1,s}$,~$\y^{0,s+1} = \y^{q+1,s}$,~$\v^{-1,s+1}=\v^{q,s},\forall s\geq1$, and~$\v^{-1,1} = \mb{0}_{np}$ from~\cref{GT-SARAH} under the vector notation.}
Under Assumption~\ref{network}, we have~\cite{matrix_analysis}$$ \mb{J} := \lim_{k\ra\infty}\mb{W}^k = \Big(\tfrac{1}{n}\mb{1}_n\mb{1}_n^\top\Big)\otimes \I_p,$$ i.e., the power limit of the network weight matrix~$\mb{W}$ is the exact averaging matrix~$\mb{J}$. 
We also introduce the following notation for convenience:
\begin{align*}
&\nabla\mb{f}(\mb{x}^{t,s}) := \left[\nabla f_1(\mb{x}_1^{t,s})^\top,\cdots,\nabla f_n(\mb{x}_n^{t,s})^\top\right]^\top, \quad \ol{\nabla\mb{f}}(\mb{x}^{t,s}) := \tfrac{1}{n}(\mb{1}_n^\top\otimes \I_p)\nabla\mb{f}(\mb{x}^{t,s}), \\
&\ol{\mb x}^{t,s} := \tfrac{1}{n}(\mb{1}_n^\top\otimes \I_p)\mb{x}^{t,s}, \qquad \ol{\mb y}^{t,s} = \tfrac{1}{n}(\mb{1}_n^\top\otimes \I_p)\mb{y}^{t,s},  \qquad 
\ol{\mb v}^{t,s} := \tfrac{1}{n}(\mb{1}_n^\top\otimes \I_p)\mb{v}^{t,s}.
\end{align*}
{\color{black}In particular, we note that~$\|\nabla\mb{f}(\mb{x}^{0,1})\|^2 := \sum_{i=1}^n\|\nabla f_i(\ol{\x}^{0,1})\|^2$.}
In the rest of the paper, we assume that~Assumptions~\ref{smooth},~\ref{sampling}, and~\ref{network} hold without explicitly stating them. 

\subsection{Auxiliary relationships}
First, as a consequence of the gradient tracking update~\eqref{gtsarah2}, it is straightforward to show by induction the following result.
\begin{lemma}\label{GT}
$\ol{\mb y}^{t+1,s} = \ol{\v}^{t,s}, \forall s\geq1$ and~$t\in[0,q]$.
\end{lemma}
\begin{proof}
See~\cref{proof_GT}.
\end{proof}
The above lemma states that the average of gradient trackers preserves the average of local gradient estimators. Under~Assumption~\ref{network}, we obtain that the weight matrix~$\mb{W}$ is a contraction operator~\cite{harnessing}.
\begin{lemma}\label{W}
$\|\mb{W}\mb{x}-\mb{J}\mb{x}\|\leq\lambda\|\mb{x}-\mb{J}\mb{x}\|$,~$\forall\mb{x}\in\mbb{R}^{np}$, for~$\lambda\in[0,1)$ defined in~\eqref{spectral}.
\end{lemma}
Lemmas~\ref{GT} and~\ref{W} are standard in decentralized optimization and gradient tracking~\cite{harnessing,DIGing}.
The~$L$-smoothness of~$F$ leads to the following quadratic upper bound~\cite{nesterov_book}:
\begin{equation}\label{DL1}
F(\mb y) \leq F(\mb x) + \langle \nabla F(\mb x), \mb{y} - \mb{x}\rangle
+ \tfrac{L}{2}\left\|\mb y - \mb x\right\|^2, \qquad\forall\mb{x},\mb{y}\in\mbb{R}^p.
\end{equation}
Consequently, the following descent type lemma on the iterates generated by \texttt{GT-SARAH} may be established by setting~$\y=\ol{\x}^{t+1,s}$ and~$\x=\ol{\x}^{t,s}$ in~\eqref{DL1} and taking a telescoping sum across all iterations of \texttt{GT-SARAH} with the help of Lemmas~\ref{GT} and~the~$L$-smoothness of each~$f_i$.
\begin{lemma}\label{DS_sum1}
If the step-size follows that~$0<\alpha\leq\frac{1}{2L}$, then we have:
\begin{align*}
\E\big[F\big(\ol{\x}^{q+1,S}\big)\big] \leq&~F\big(\ol{\x}^{0,1}\big) - \frac{\a}{2}\sum_{s=1}^S\sum_{t=0}^{q}\E\Big[\big\|\nabla F(\ol{\x}^{t,s})\big\|^2\Big]
- \frac{\a}{4}\sum_{s=1}^S\sum_{t=0}^{q}\E\big[\|\ol{\v}^{t,s}\|^2\big] \\
&+ \a\sum_{s=1}^S\sum_{t=0}^{q}\E\Big[\big\|\ol{\v}^{t,s}\!-\ol{\nabla\mb{f}}(\x^{t,s})\big\|^2\Big] + \a L^2\sum_{s=1}^S\sum_{t=0}^{q}{\E\bigg[\frac{\|\mb{x}^{t,s}-\mb{J}\mb{x}^{t,s}\|^2}{n}\bigg]}.
\end{align*}
\end{lemma}
\begin{proof}
See~\cref{proof_DS_sum1}.
\end{proof}
In light of Lemma~\ref{DS_sum1}, our analysis approach is to derive the range of the step-size~$\alpha$ of~\GS~such that
\begin{align*}
\frac{1}{4}\sum_{s=1}^S\sum_{t=0}^{q}\E\big[\|\ol{\v}^{t,s}\|^2\big] 
- \sum_{s=1}^S\sum_{t=0}^{q}\E\Big[\big\|\ol{\v}^{t,s}-\ol{\nabla\mb{f}}(\x^{t,s})\big\|^2\Big] 
- L^2\sum_{s=1}^S\sum_{t=0}^{q}\E\bigg[\frac{\|\mb{x}^{t,s}-\mb{J}\mb{x}^{t,s}\|^2}{n}\bigg]
\end{align*}
is non-negative and therefore establishes the convergence of~\GS~to a first-order stationary point following the standard arguments in \emph{batch} gradient descent for non-convex problems~\cite{nesterov_book,OPT_ML}. 
To this aim, we need to derive upper bounds for two error terms in the above expression: (i)~$\|\ol{\v}^{t,s}-\ol{\nabla\mb{f}}(\x^{t,s})\|^2$, the gradient estimation error; and (ii)~${\|\mb{x}^{t,s}-\mb{J}\mb{x}^{t,s}\|^2}$, the state consensus error. We quantify these two errors next and then return to Lemma~\ref{DS_sum1}. The following lemma is obtained with similar probabilistic arguments for~\SARAH-type~\cite{SPIDER,spiderboost,sarah_ncvx} estimators, however, with subtle modifications due to the decentralized network effect.
\begin{lemma}\label{vrlem}
We have:~${\forall s\geq1}$,
\begin{align*}
\sum_{t=0}^{q}\E\Big[\big\|\ol{\v}^{t,s} - \ol{\nabla\mb{f}}(\x^{t,s})\big\|^2\Big] 
\leq\frac{3q\a^2L^2}{nB}\sum_{t=0}^{q-1}\E\big[\|\ol{\v}^{t,s}\|^2\big] + \frac{6qL^2}{nB}\sum_{t=0}^{q}\E\bigg[\frac{\|\x^{t,s} - \mb{J}\x^{t,s}\|^2}{n}\bigg].
\end{align*}
\end{lemma}
\begin{proof}
See~\cref{proof_vrlem}.
\end{proof}

Note that Lemma~\ref{vrlem} shows that the accumulated gradient estimation error over one inner loop may be bounded by the accumulated state consensus error and the norm of the gradient estimators. Lemma~\ref{vrlem} thus may be used to simplify the right hand side of
the descent inequality in Lemma~\ref{DS_sum1}. Naturally, what is left is to seek an upper bound for the state consensus error in terms of~$\E[\|\ol{\v}^{t,s}\|^2]$.
This result is presented in the following lemma.

\begin{lemma}\label{consensus_bound1}
If the step-size follows~${0<\alpha\leq\frac{(1-\lambda^2)^2}{8\sqrt{42}L}}$, then 
\begin{align*}
\sum_{s=1}^S\sum_{t=0}^q\mbb{E}\bigg[\frac{\|\mb{x}^{t,s}-\mb{J}\mb{x}^{t,s}\|^2}{n}\bigg]\leq\frac{64\a^2}{(1-\lambda^2)^3}\frac{\|\nabla\mb{f}(\mb{x}^{0,1})\|^2}{n}
+ \frac{1536\a^4L^2}{(1-\lambda^2)^4}\sum_{s=1}^S\sum_{t=0}^q\mathbb{E}\big[\|\ol{\v}^{t,s}\|^2\big].
\end{align*}
\end{lemma}
\begin{proof}
See~\cref{proof_consensus_bound1}.
\end{proof}
Establishing Lemma~\ref{consensus_bound1} requires a careful analysis; here, we provide a brief sketch. Recall the \GS~algorithm in~\eqref{gtsarah1}-\eqref{gtsarah2} and note that the state vector~$\mb x^{t,s}$ is coupled with the gradient tracker~$\mb y^{t,s}$. Thus, in order to quantify the state consensus error~$\|\mb x^{t,s} - \mb{J} \mb x^{t,s}\|^2$, we need to establish its relationship with the gradient tracking error~$\|\mb y^{t,s} - \mb{J} \mb y^{t,s}\|^2$. In fact, we show that these coupled errors jointly formulate a linear time-invariant (LTI) system dynamics whose system matrix is stable under a certain range of the step-size~$\a$. Solving this LTI yields Lemma~\ref{consensus_bound1}.

Finally, it is straightforward to use Lemmas~\ref{vrlem} and~\ref{consensus_bound1} to refine the descent inequality in Lemma~\ref{DS_sum1} to obtain the following result.
\begin{lemma}\label{DS_sum2}
If~$0<\alpha\leq\ol{\a}:=\min\left\{\frac{(1-\lambda^2)^2}{4\sqrt{42}},\big(\frac{nB}{6q}\big)^{\sfrac{1}{2}},\big(\frac{4nB}{7nB+24q}\big)^{\sfrac{1}{4}}\frac{1-\lambda^2}{6}\right\}\frac{1}{2L}$, then 
\begin{align*}
L^2\sum_{s=1}^S\sum_{t=0}^{q}\E\bigg[\frac{\|\mb{x}^{t,s}-\mb{J}\mb{x}^{t,s}\|^2}{n}&\bigg]
+ \frac{1}{n}\sum_{i=1}^n\sum_{s=1}^S\sum_{t=0}^{q}\E\Big[\big\|\nabla F\big(\x_i^{t,s}\big)\big\|^2\Big] \n\\
\leq&~\frac{4(F\big(\ol{\x}^{0,1}\big) - F^*)}{\a} 
+ \left(\frac{7}{4}+\frac{6q}{nB}\right) \frac{256\a^2L^2}{(1-\lambda^2)^3}\dfrac{\|\nabla\mb{f}(\mb{x}^{0,1})\|^2}{n}.
\end{align*}
\end{lemma}
\begin{proof}
See~\cref{proof_DS_sum2}.
\end{proof}
We note that the descent inequality in~\cref{DS_sum2} that characterizes the convergence of \textbf{\texttt{GT-SARAH}} is independent of the variance of local gradient estimators and of the difference between the local and the global gradient. In fact, it has similarities to that of the centralized~\emph{batch} gradient descent~\cite{nesterov_book,OPT_ML}; see also the discussion on \texttt{DSGD} in Section~\ref{sec:intro}. This is a consequence of the joint use of the local variance reduction and the global gradient tracking. This is essentially why we are able to match the gradient complexity of the centralized near-optimal methods for finite sum problems and obtain the almost sure convergence guarantee of \texttt{GT-SARAH} to a stationary point. 

\subsection{Proofs of the main theorems}\label{Proof_main}
With the refined descent inequality in Lemma \ref{DS_sum2} at hand, Theorems~\ref{asp},~\ref{out}, and~\ref{main} are now straightforward to prove.
\begin{T1}
We observe from \cref{DS_sum2} that if~$0<\a\leq\ol{\a}$, then $\sum_{s=1}^{\infty}\sum_{t=0}^{q}\E\big[\|\nabla F(\x_i^{t,s})\|^2 + L^2\|\x_i^{t,s}-\ol{\x}^{t,s}\|^2\big]<\infty,\forall i\in\mc{V},$ which implies all nodes achieve consensus and converge to a stationary point in the mean-squared sense. Further, by monotone convergence theorem~\cite{probability_with_martinagles}, we exchange the order of the expectation and the series to obtain:~$
\E\big[\sum_{s=1}^{\infty}\sum_{t=0}^{q}\big(\|\nabla F(\x_i^{t,s})\|^2 + L^2\|\x_i^{t,s}-\ol{\x}^{t,s}\|^2\big)\big]< \infty$, $\forall i\in\mc{V}$, which leads to ${\P\big(\sum_{s=1}^{\infty}\sum_{t=0}^{q}\big(\|\nabla F(\x_i^{t,s})\|^2 + L^2\|\x_i^{t,s}-\ol{\x}^{t,s}\|^2\big)<\infty\big)=1},\forall i\in\mc{V}$,
i.e., the consensus and convergence to a stationary point in the almost sure sense.
\end{T1}

\begin{T2}
We recall the metric of the outer loop complexity in Definition~\ref{FSL} and we divide the descent inequality in \cref{DS_sum2} by~$S(q+1)$ from both sides. It is then clear that to find an~$\epsilon$-accurate stationary point of~$F$, it suffices to choose the total number of the outer loop iterations~$S$ such that 
\begin{align}\label{S_sol}
\frac{4\left(F(\ol{\x}^{0,1}) - F^*\right)}{S(q+1)\a} 
+ \left(\frac{7}{4}+\frac{6q}{nB}\right)\frac{256\a^2L^2}{S(q+1)(1-\lambda^2)^3}\frac{\|\nabla\mb{f}(\mb{x}^{0,1})\|^2}{n} \leq\epsilon^2.
\end{align}
The proof follows by that if~$0<\a\leq\big(\frac{4nB}{7nB+24q}\big)^{\sfrac{1}{3}}\frac{1-\lambda^2}{12L}$, then~$(\frac{7}{4}+\frac{6q}{nB})\frac{256\a^2L^2}{(1-\lambda^2)^3}\leq\frac{1}{\a L}$, and by solving for the lower bound on~$S$ such that~\eqref{S_sol} holds.
\end{T2}

\begin{T3}
During each inner loop, \texttt{GT-SARAH} incurs~${n(m + 2qB)}$ component gradient computations across all nodes and~$q$ rounds of communication of the network. Hence, to find an~$\epsilon$-accurate stationary point of~$F$, \texttt{GT-SARAH} requires, according to \cref{out}, at most 
$$\mc{H} = \mc{O}\bigg(\frac{n(m+qB)}{q\a L\epsilon^2}\bigg(L\left(F\big(\ol{\x}^{0,1}\big) - F^*\right)
+ \frac{\|\nabla\mb{f}(\mb{x}^{0,1})\|^2}{n}\bigg)$$
component gradient computations across all nodes and
$$\mc{K} = \mc{O}\bigg(\frac{1}{\a L\epsilon^2}\bigg(L\left(F\big(\ol{\x}^{0,1}\big) - F^*\right)
+ \frac{\|\nabla\mb{f}(\mb{x}^{0,1})\|^2}{n}\bigg)\bigg)$$
rounds of communication of the network. The proof follows by setting the step-size~$\a$ as its upper bound in \cref{out} and
the length of the inner loop as $q = \mc{O}(\frac{m}{B})$.
\end{T3}

\section{Numerical experiments}\label{numer}
{\color{black}In this section, we illustrate, by numerical experiments, our main theoretical claim that \texttt{GT-SARAH} finds a first-order stationary point of Problem~\eqref{P}
with a significantly improved gradient complexity compared to the existing decentralized stochastic gradient methods.

\subsection{Setup}
We consider a \emph{non-convex} logistic regression model~\cite{LR_NCVX} for binary classification over a decentralized network of~$n$ nodes with~$m$ data samples at each node: $\min_{\x\in\mathbb{R}^p}F(\mb{x}) 
:= \frac{1}{n}\sum_{i=1}^{n}\frac{1}{m}\sum_{j=1}^{m}\big(f_{i,j}(\x)
+ r(\x)\big),$ such that the logistic loss~$f_{i,j}(\x)$ and the non-convex regularization~$r(\x)$ are given by
\begin{align}\label{lr_model}
f_{i,j}(\x) := \log\Big[1+\exp\big\{\!-(\mb{x}^\top\bds{\theta}_{i,j})\xi_{i,j}\big\}\Big]
~
\text{and}
~
r(\x) := R\tsum_{d=1}^p[\x]_d^2\big(1+[\x]_d^2\big)^{-1},
\end{align}
where~$[\x]_d$ denotes the~$d$-th coordinate of~$\x$.
In~\eqref{lr_model}, note that~${\bds{\theta}_{i,j}\in\mathbb{R}^{p}}$ is the~$j$-th data sample at the~$i$-th node and~${\xi_{i,j}\in\{-1,+1\}}$ is the corresponding binary label. The details of the datasets under consideration are provided in Table~\ref{datasets}. We normalize each data sample such that~$\|\bds{\theta}_{i,j}\| = 1, \forall i,j$, and set the regularization parameter as~$R = 10^{-3}$. The primitive doubly stochastic weight matrices associated with the networks are generated by the lazy Metroplis rule~\cite{tutorial_nedich}. We characterize the performance of the algorithms in
comparison in terms of the decrease of the network stationary gap versus epochs, where the stationary gap is defined as~$\|\nabla F(\ol{\x})\| + \frac{1}{n}\sum_{i=1}^n\|\x_i-\ol{\x}\|$, where~$\x_i$ is the estimate of the stationary point of~$F$ at node~$i$ and~$\ol{\x} := \frac{1}{n}\sum_{i=1}^n\x_i$, and each epoch represents~$m$ component gradient computations at each node. 

\subsection{Performance comparisons} We compare the performance of \texttt{GT-SARAH} with \texttt{DSGT}~\cite{improved_DSGT_Xin} and \texttt{D-GET}~\cite{D_Get}; we note that \texttt{D2}~\cite{D2} and \texttt{DSGD}~\cite{DSGD_nips} are not presented here for conciseness, since in general the former achieves a similar performance with \texttt{DSGT} and the latter underperforms \texttt{DSGT} and \texttt{D2}~\cite{SPM_Xin,SPM_Hong,D2}. 
Towards the parameter selection of each algorithm, we use the following setup: (i) for \texttt{GT-SARAH}, we choose its minibatch size as~$B = 1$ and its inner-loop length as~$q = m$ in light of~\cref{main1}; (ii) for \texttt{D-GET}, we choose its minibatch size and inner-loop length as~$\floor{m^{\sfrac{1}{2}}}$ under which its convergence is established; see Theorem~$1$ in~\cite{D_Get}; and (iii) we manually optimize the step-sizes for \texttt{GT-SARAH}, \texttt{D-GET}, and \texttt{DSGT} across all experiments.

\begin{table}[hbt]
\footnotesize
\renewcommand{\arraystretch}{1.1}
\caption{Datasets used in numerical experiments, available at \href{https://www.openml.org/}{https://www.openml.org/}.}
\vspace{-0.2cm}
\begin{center}
\begin{tabular}{|c|c|c|c|}
\hline
\textbf{Dataset} & \textbf{Number of samples} ($N = nm$)  & \textbf{dimension} ($p$) \\ \hline
covertype &  $100,\!000$ & $54$ \\ \hline
MiniBooNE &  $100,\!000$ & $51$ \\ \hline
KDD98 &  $82,\!000$ & $477$ \\ \hline
w8a & $60,\!000$ & $300$  \\ \hline
a9a & $48,\!800$ & $124$  \\ \hline
Fashion-MNIST (T-shirt versus dress) & $10,\!000$ & $784$  \\ \hline
\end{tabular}
\end{center}
\label{datasets}
\vspace{-0.2cm}
\end{table}

We first compare the performances of \texttt{GT-SARAH}, \texttt{DSGT}, and \texttt{D-GET} in the big-data regime, that is, the number of samples~$m$ at each node is relatively large. To this aim, we distribute the covertype, MiniBooNE, and KDD98 dataset
over a~$10$-node exponential graph~\cite{tutorial_nedich} whose associated second largest singular value~$\lambda \approx 0.71$. The experimental results are presented in~\cref{figure_comp_bigdata}, where \texttt{GT-SARAH} outperforms \texttt{DSGT} and \texttt{D-GET}. We also observe that \texttt{D-GET} outperforms \texttt{DSGT} in this case since the performance of the latter is deteriorated by the large variance of the stochastic gradients as the number of the samples~$m$ at each node is large.   

\begin{figure}[!ht]
\centering
\includegraphics[width=1.65in]{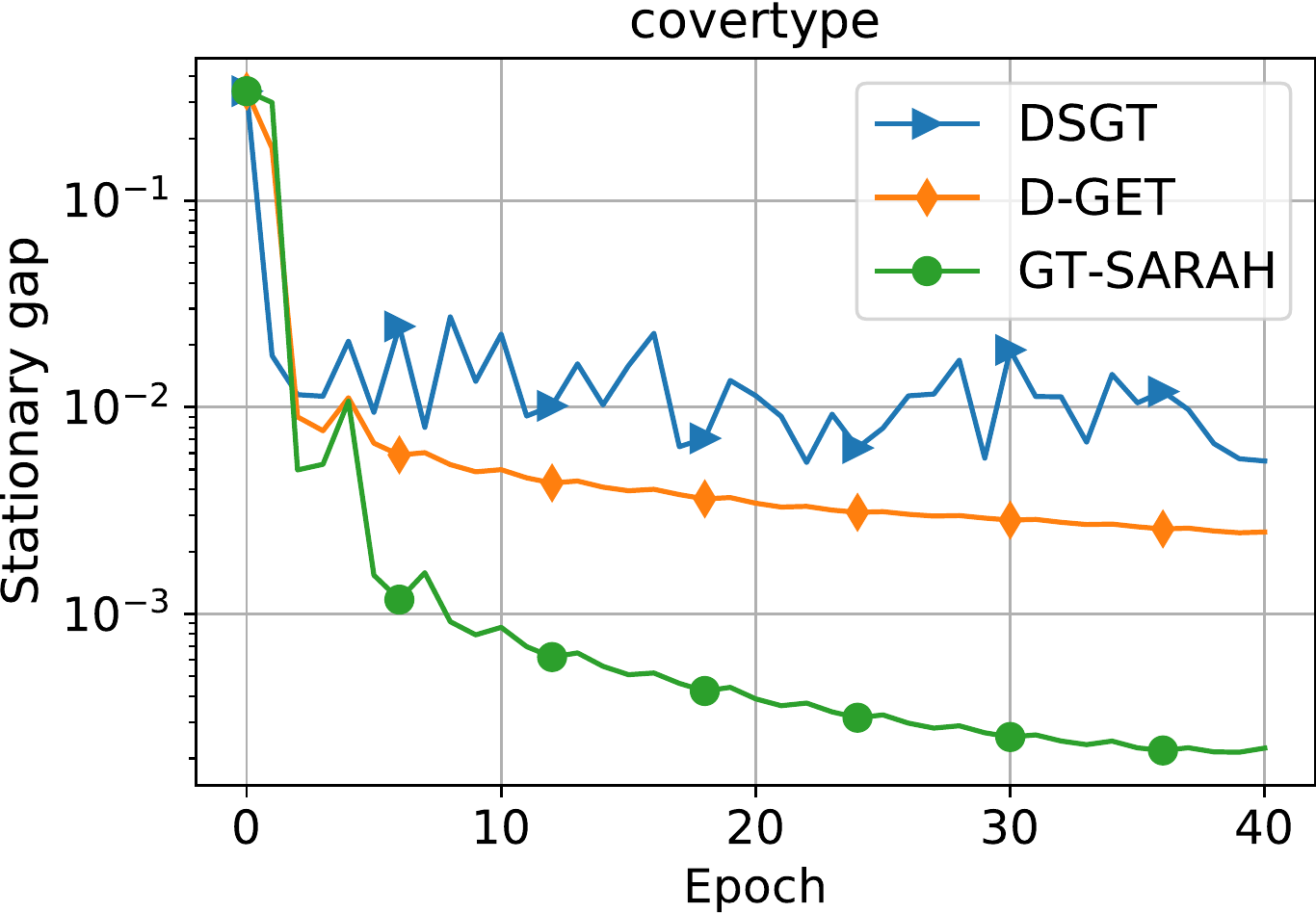}~
\includegraphics[width=1.65in]{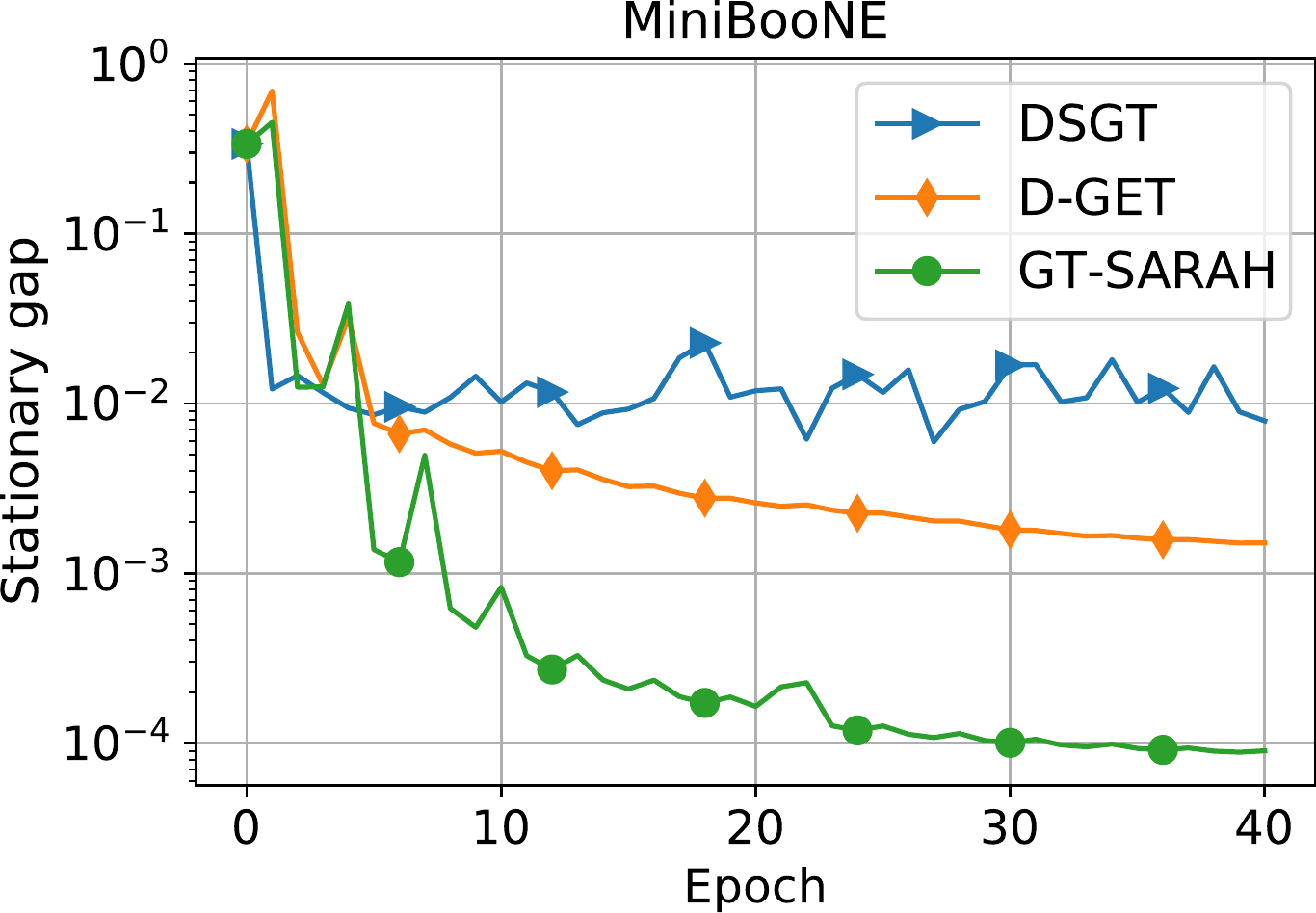}~
\includegraphics[width=1.65in]{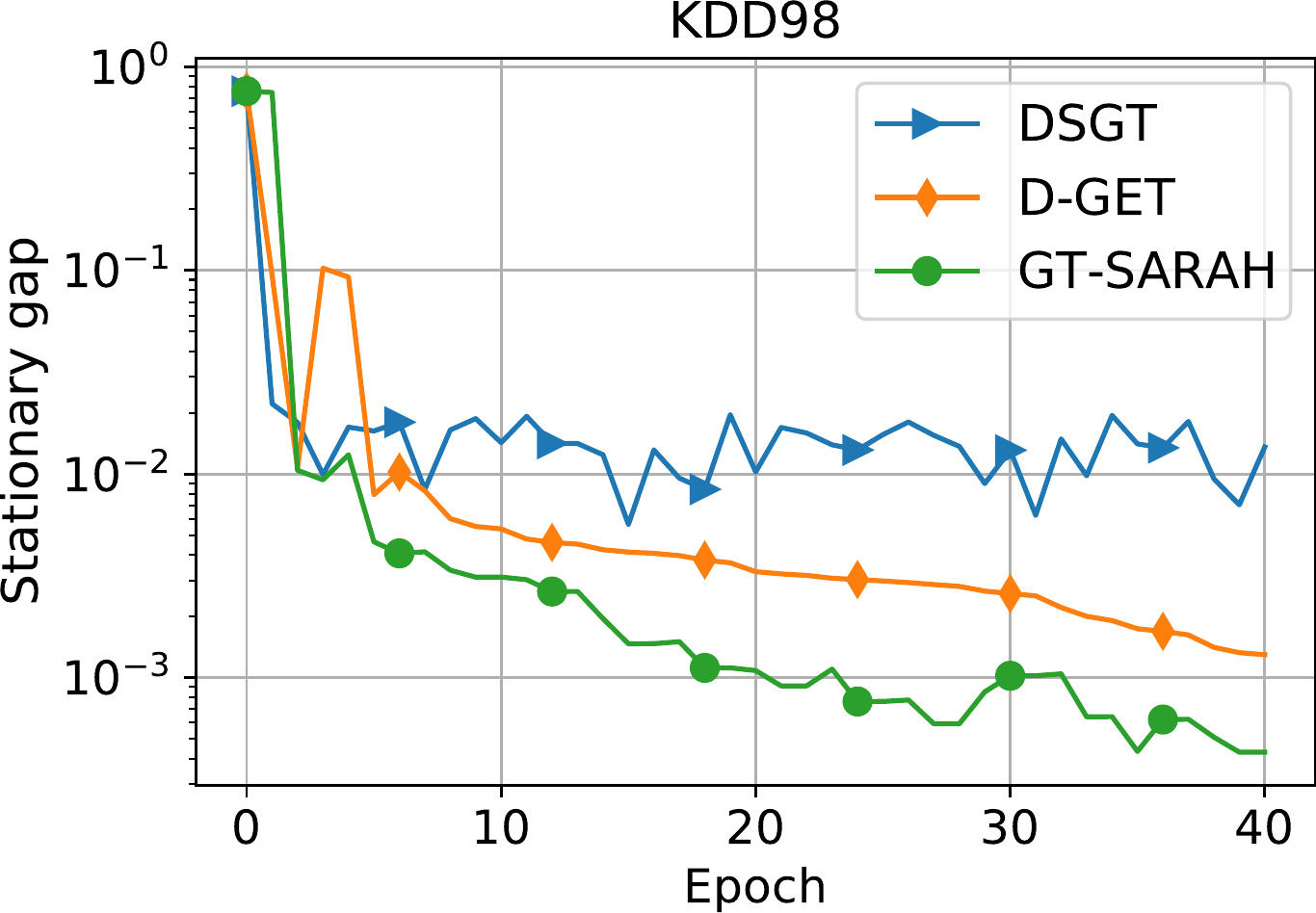}
\vspace{-0.7cm}
\caption{Performance comparison of \texttt{GT-SARAH}, \texttt{DSGT}, and \texttt{D-GET} over a~$10$-node exponential graph on the covertype, MiniBooNE, and KDD98 dataset.}
\label{figure_comp_bigdata}
\end{figure}

We next consider the large-scale network regime, where the network spectral gap inverse~$(1-\lambda)^{-1}$ and the number of the nodes~$n$ are relatively large compared with the local sample size~$m$. We distribute the w8a, a9a, and Fashion-MNIST dataset over the~$n = 10\times 10$ grid graph whose associated second largest eigenvalue~$\lambda \approx 0.99$. The performance comparison of the algorithms is shown in~\cref{figure_comp_bignet}, where we observe that \texttt{GT-SARAH} still outperforms \texttt{DSGT} and \texttt{D-GET}. Besides, it is worth noting that \texttt{D-GET} underperforms \texttt{DSGT} in this case. We provide an explanation about this phenomenon in the following. In the regime where~$m$ is relatively small, the variance of the stochastic gradients is relatively small and as a consequence \texttt{DSGT} performs well. 
On the other hand, the minibatch size~$\floor{m^{\sfrac{1}{2}}}$ of \texttt{D-GET} is too large in this regime to achieve a satisfactory performance; see the related discussion in Subsection~\ref{two_regimes}. 

\begin{figure}[!ht]
\centering
\includegraphics[width=1.65in]{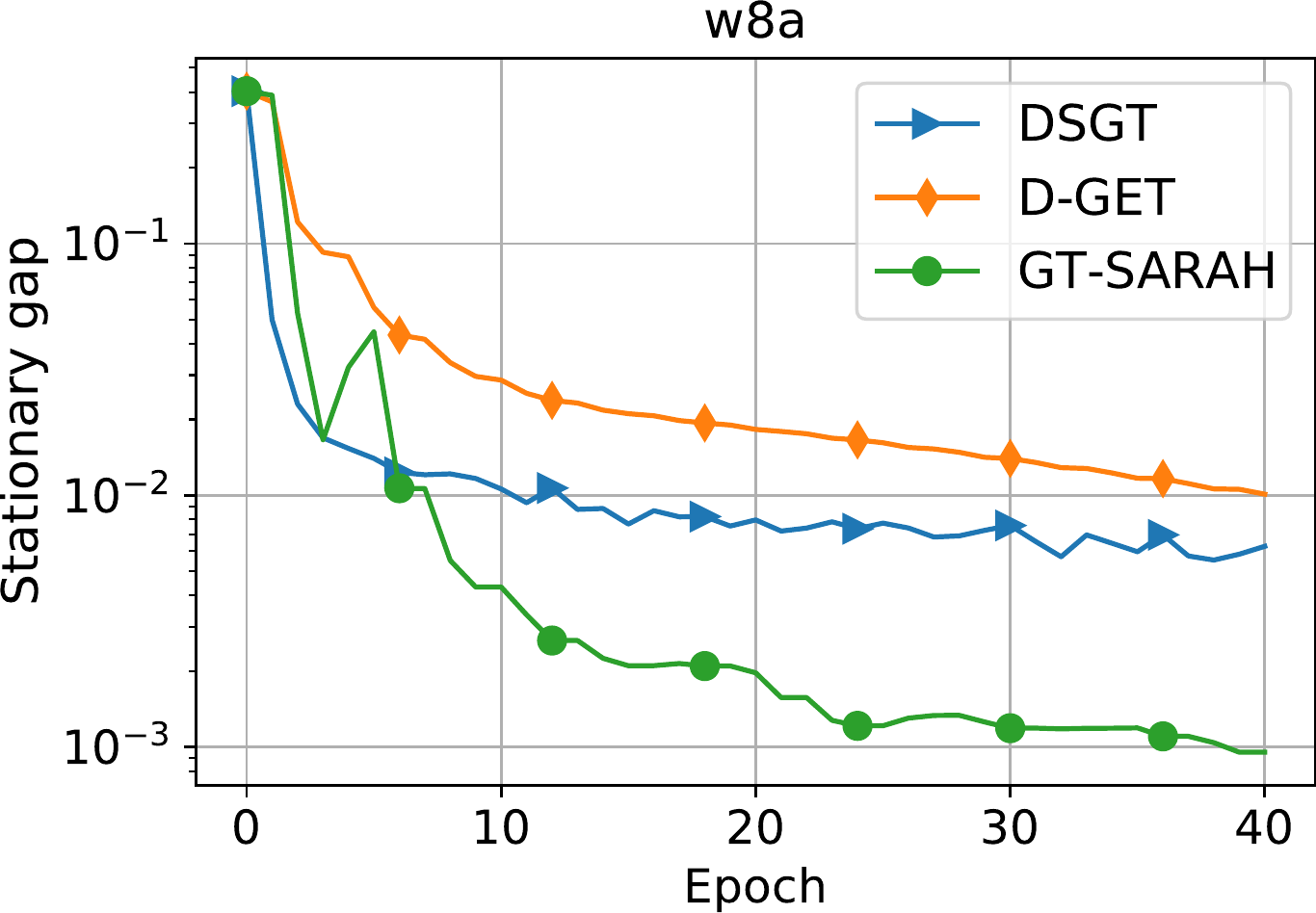}~
\includegraphics[width=1.65in]{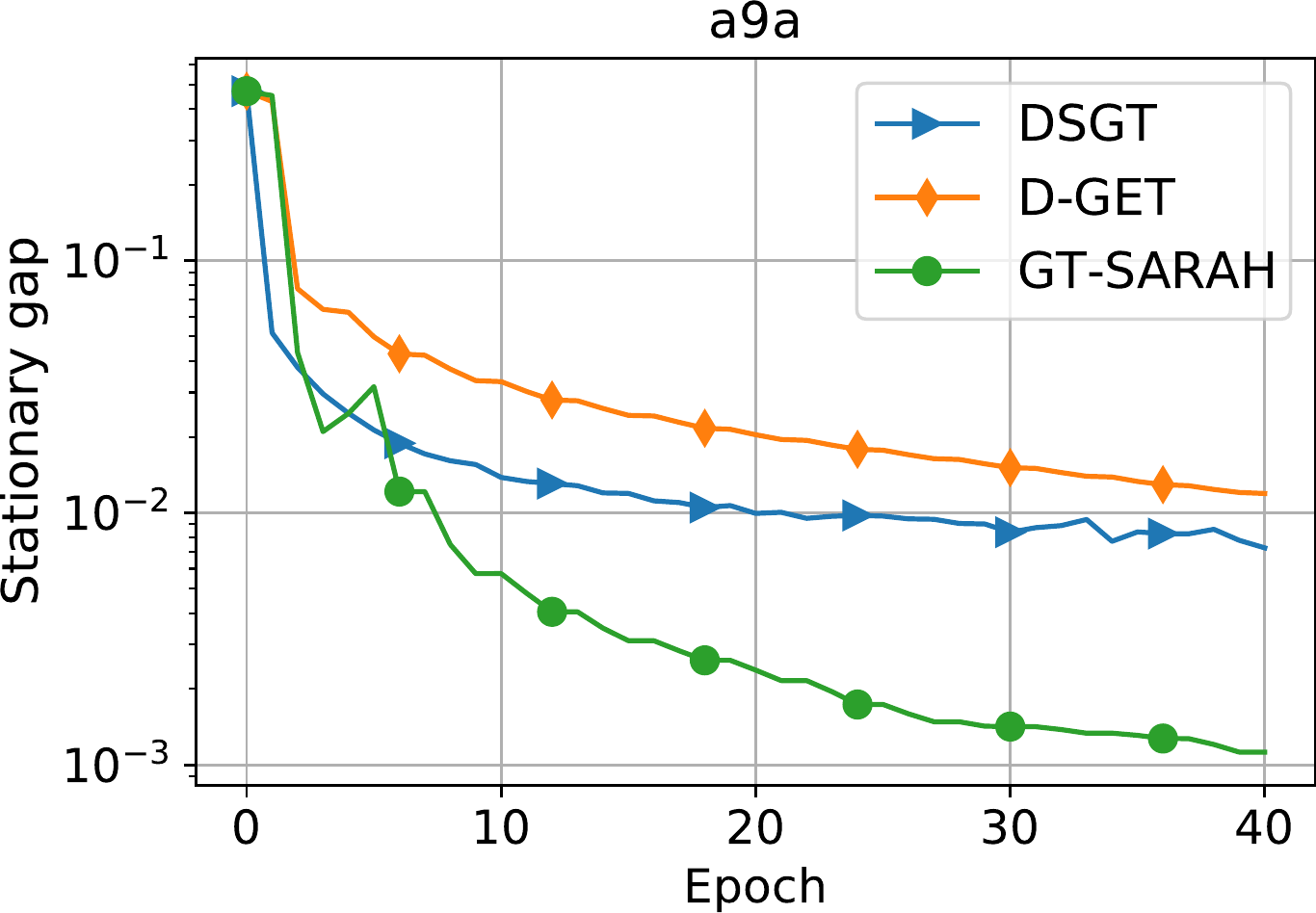}~
\includegraphics[width=1.65in]{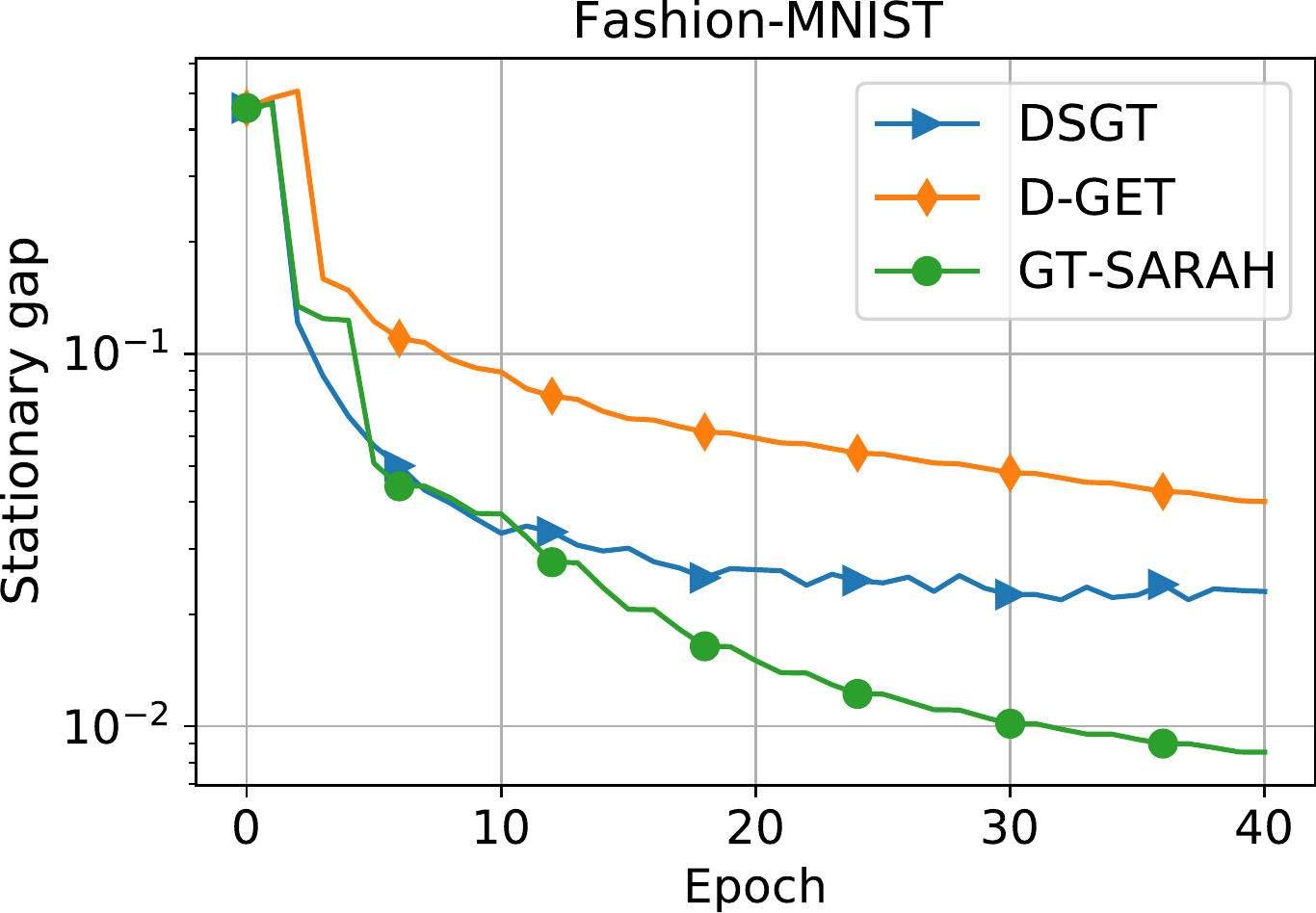}
\vspace{-0.7cm}
\caption{Performance comparison of \texttt{GT-SARAH}, \texttt{DSGT}, and \texttt{D-GET} over the~$10\times10$ grid graph on the w8a, a9a, and Fashion-MNIST dataset.}
\label{figure_comp_bignet}
\end{figure}}

\vspace{-0.3cm}
\section{Conclusions}\label{concl}
In this paper, we propose \texttt{GT-SARAH} to minimize a finite-sum of~$N$ smooth non-convex functions, equally distributed over a decentralized network of~$n$ nodes. With appropriate algorithmic parameters, \texttt{GT-SARAH} achieves significantly improved gradient complexity compared with the existing decentralized stochastic gradient methods. In particular, in a big-data regime~${n=\mc{O}(N^{\sfrac{1}{2}}(1-\lambda)^3)}$, the gradient complexity of \texttt{GT-SARAH} reduces to~${\mc{O}(N^{\sfrac{1}{2}}L\e^{-2})}$ which matches that of the centralized near-optimal variance-reduced methods such as \texttt{SPIDER}~\cite{SPIDER,spiderboost} and~\texttt{SARAH}~\cite{sarah_ncvx} for Problem~\eqref{P}, where~$L$ is the smoothness parameter and~${(1-\lambda)}$ is the spectral gap of the network weight matrix. Moreover, \texttt{GT-SARAH} in this regime achieves non-asymptotic linear speedup compared with the centralized near-optimal approaches that perform all gradient computations at a single node. Compared with the minibatch implementations of \texttt{SPIDER} and~\texttt{SARAH} over server-worker architectures~\cite{DVR_master_worker}, the decentralized \texttt{GT-SARAH} enjoys the same non-asymptotic linear speedup in terms of the gradient complexity, however, admits sparser and more flexible communication topology. 

\appendix

\section{Preliminaries of the convergence analysis}
In this section, we present the preliminaries for the proofs of the technical lemmas~\ref{GT},~\ref{DS_sum1},~\ref{vrlem},~\ref{consensus_bound1},~\ref{DS_sum2}. We first define the natural filtration associated with the probability space, an increasing family of sub-$\sigma$-algebras of~$\mc{F}$, as 
\begin{align*}
\mc{F}^{t,s} := \sigma\left(\sigma\big(\tau_{i,l}^{t-1,s}:i\in\mc{V},l\in[1,B]\big),~\mc{F}^{t-1,s}\right),\quad t\in[2,q+1],~s\geq 1,
\end{align*}
where~$\mc{F}^{1,s} :=: \mc{F}^{0,s} := \mc{F}^{q+1,s-1}, s\geq 2,$ and $\mc{F}^{1,1} :=: \mc{F}^{0,1} := \{\phi,\Xi\}$.
It can be verified by induction that~$\mb{x}^{t,s}$,~$\mb{y}^{t,s}$ are~$\mc{F}^{t,s}$-measurable, and $\mb{v}^{t,s}$ is $\mc{F}^{t+1,s}$-measurable,~$\forall s\geq 1$ and~$t\in[0,q]$.
We assume that the starting point~$\ol{\mb{x}}^{0,1}$ of~\GS~is a constant vector.
We next present some standard results in the context of decentralized optimization and gradient tracking methods.
The following lemma provides an upper bound on the difference between the exact global gradient and the average of local batch gradients in terms of the state consensus error, as a result of the~$L$-smoothness of each~$f_i$.
\begin{lemma}\label{Lbound}
$\b\|\ol{\nabla\mb{f}}(\mb{x}^{t,s})-\nabla F(\ol{\mb{x}}^{t,s})\b\|^2\leq\frac{L^2}{n}\b\|\mb{x}^{t,s}-\mb{J}\mb{x}^{t,s}\b\|^2$,~$\forall s\geq1$ and~$t\in[0,q]$.
\end{lemma}
\begin{proof}
Observe that:~$\forall s\geq1$ and~$t\in[0,q]$,
\begin{align*}
\big\|\ol{\nabla\mb{f}}(\mb{x}^{t,s})-\nabla F\b(\ol{\mb{x}}^{t,s}\b)\big\|^2
=&~\tfrac{1}{n^2}\big\|\tsum_{i=1}^n\big(\nabla f_i\big(\mb{x}_i^{t,s}\big)-\nabla f_i\b(\ol{\x}^{t,s}\b)\big)\big\|^2 
\nonumber\\
=&~\tfrac{1}{n}\tsum_{i=1}^n\big\|\nabla f_i\big(\mb{x}_i^{t,s}\big)-\nabla f_i\b(\ol{\x}^{t,s}\b)\big\|^2 
\nonumber\\
\leq&~\tfrac{L^2}{n}\tsum_{i=1}^n\big\|\mb{x}_i^{t,s}-\ol{\x}^{t,s}\big\|^2,
\end{align*}
where the last line is due to the~$L$-smoothness of each~$f_i$. The proof is complete.
\end{proof}
The following are some standard inequalities on the state consensus error.
\begin{lemma}\label{contraction_consensus}
The following inequalities holds:~$\forall s\geq1$ and~$t\in[0,q]$,
\begin{align}
&\big\|\mb{x}^{t+1,s}-\mb{J}\mb{x}^{t+1,s}\big\|^2 \leq \tfrac{1+\lambda^2}{2} \big\|\mb{x}^{t,s}-\mb{J}\mb{x}^{t,s}\big\|^2 + \tfrac{2\alpha^2}{1-\lambda^2}\big\|\mb{y}^{t+1,s}-\mb{J}\mb{y}^{t+1,s}\big\|^2.\label{consensus1}\\
&\big\|\mb{x}^{t+1,s}-\mb{J}\mb{x}^{t+1,s}\big\|^2 \leq 2 \big\|\mb{x}^{t,s}-\mb{J}\mb{x}^{t,s}\big\|^2 + 2\alpha^2\big\|\mb{y}^{t+1,s}-\mb{J}\mb{y}^{t+1,s}\big\|^2.  \label{consensus2}
\end{align}
\end{lemma}
\begin{proof}
Using~\eqref{gtsarah2} and the fact that~$\J\W = \J$, we have:~$\forall s\geq1$ and~$\forall t\in[0,q]$,
\begin{align}
\big\|\mb{x}^{t+1,s}-\J\mb{x}^{t+1,s}\big\|^2 
=&~ \big\|\W\mb{x}^{t,s} - \alpha\mb{y}^{t+1,s}-\J\big(\W\mb{x}^{t,s} - \alpha\mb{y}^{t+1,s}\big)\big\|^2 \nonumber\\
=&~ \big\|\W\mb{x}^{t,s} - \J\mb{x}^{t,s}-\alpha\big(\mb{y}^{t+1,s}-\J\mb{y}^{t+1,s}\big)\big\|^2
\label{consensus0}
\end{align}
We apply Young's inequality, $\|\mb{a}+\mb{b}\|^2\leq(1+\eta)\|\mb{a}\|^2+(1+\eta^{-1})\|\mb{b}\|^2,\forall\mb{a},\mb{b}\in\mathbb{R}^{np}$, $\forall\eta>0$, and Lemma~\ref{W} to~\eqref{consensus0} to obtain:~$\forall s\geq1$ and~$\forall t\in[0,q]$,
\begin{align*}
\big\|\mb{x}^{t+1,s}-\J\mb{x}^{t+1,s}\big\|^2
\leq&~\left(1+\eta\right)\lambda^2\big\|\mb{x}^{t,s} - \J\mb{x}^{t,s}\big\|^2
+\left(1+\eta^{-1}\right)\alpha^2\big\|\mb{y}^{t+1,s} - \J\mb{y}^{t+1,s}\big\|^2.
\end{align*}
Setting~$\eta$ as $\frac{1-\lambda^2}{2\lambda^2}$ and~$1$  respectively yields~\eqref{consensus1} and~\eqref{consensus2}. 
\end{proof}

\subsection{Proof of Lemma~\ref{GT}}\label{proof_GT}
Recall from Assumption~\ref{network} that~$\mb 1_n^\top \ul{\W} = \mb 1_n^\top$. We multiply~\eqref{gtsarah1} by~$\ave$ to obtain:~$\forall s\geq1$ and~$t\in[0,q]$,
\begin{align*}
\ol{\mb y}^{t+1,s} 
=&~\ol{\mb y}^{t,s} + \ol{\mb v}^{t,s} - \ol{\mb v}^{t-1,s} \\
=&~\ol{\mb y}^{t-1,s} + \ol{\mb v}^{t,s} - \ol{\mb v}^{t-2,s} \\
\cdots\\
=&~\ol{\mb y}^{0,s} + \ol{\mb v}^{t,s} - \ol{\mb v}^{-1,s} \\
=&~\ol{\mb y}^{q+1,s-1} + \ol{\mb v}^{t,s} - \ol{\mb v}^{q,s-1}\\
\cdots\\
=&~\ol{\mb y}^{0,1} + \ol{\mb v}^{t,s} - \ol{\mb v}^{-1,1} = \ol{\mb v}^{t,s},
\end{align*}
where the above series of equalities follows directly from the updates of~\GS.

\section{Proof of Lemma~\ref{DS_sum1}}\label{proof_DS_sum1}
We multiply~\eqref{gtsarah2} by~$\ave$ and then use Lemma~\ref{GT} to obtain the recursion of the mean state~$\ol{\mb{x}}^{t,s}$ as follows:
\begin{align*}
\ol{\mb{x}}^{t+1,s} 
= \ol{\mb{x}}^{t,s} - \alpha\ol{\mb y}^{t+1,s}
= \ol{\mb{x}}^{t,s} - \alpha\ol{\mb v}^{t,s},\qquad\forall s\geq1~\mbox{and}~t\in[0,q].
\end{align*}
Setting~$\y=\ol{\x}^{t+1,s}$ and~$\x=\ol{\x}^{t,s}$ in~\eqref{DL1}, we have:~$\forall s\geq1$ and $t\in[0,q]$,
\begin{align}\label{dl2}
F(\ol{\x}^{t+1,s}) \leq F(\ol{\x}^{t,s}) - \alpha\big\langle\nabla F(\ol{\x}^{t,s}),\ol{\v}^{t,s}\big\rangle
+ \tfrac{\alpha^2 L}{2}\b\|\ol{\v}^{t,s}\b\|^2.
\end{align}
Applying~$\langle \mb{a},\mb{b} \rangle = 0.5\left(\|\mb a\|^2 + \|\mb b\|^2 - \|\mb{a}-\mb{b}\|^2\right), \forall\mb{a},\mb{b}\in\mbb{R}^p$, to~\cref{dl2}, we obtain 
an inequality that characterizes the descent of the network mean state over one inner loop iteration:~$\forall s\geq1$ and~$t\in[0,q]$,
\begin{align}
F(\ol{\x}^{t+1,s}) \leq&~F(\ol{\x}^{t,s}) - \tfrac{\a}{2}\b\|\nabla F(\ol{\x}^{t,s})\b\|^2
- \tfrac{\alpha(1 - \alpha L)}{2}\b\|\ol{\v}^{t,s}\b\|^2 + \tfrac{\a}{2}\b\|\ol{\v}^{t,s}-\nabla F(\ol{\x}^{t,s})\b\|^2, \nonumber\\
\leq&~F(\ol{\x}^{t,s}) - \tfrac{\a}{2}\b\|\nabla F(\ol{\x}^{t,s})\b\|^2
- \tfrac{\a}{4}\b\|\ol{\v}^{t,s}\b\|^2 + \a\b\|\ol{\v}^{t,s}-\ol{\nabla\mb{f}}(\x^{t,s})\b\|^2 \n\\
&+ \a\b\|\ol{\nabla\mb{f}}(\x^{t,s})-\nabla F(\ol{\x}^{t,s})\b\|^2
\label{L30}\\
\leq&~F(\ol{\x}^{t,s}) - \tfrac{\a}{2}\b\|\nabla F(\ol{\x}^{t,s})\b\|^2
- \tfrac{\a}{4}\b\|\ol{\v}^{t,s}\b\|^2 + \a\b\|\ol{\v}^{t,s}-\ol{\nabla\mb{f}}(\x^{t,s})\b\|^2 \n\\
&+ \tfrac{\a L^2}{n}\b\|\mb{x}^{t,s}-\mb{J}\mb{x}^{t,s}\b\|^2, \label{L300}
\end{align}
where~\cref{L30} is due to~$0<\a\leq\frac{1}{2L}$ and~\cref{L300} is due to Lemma~\ref{Lbound}. We then take the telescoping sum of~\eqref{L300} over~$t$ from~$0$ to~$q$ to obtain:~$\forall s\geq1$,
\begin{align}\label{L31}
F(\ol{\x}^{0,s+1}) 
\leq&~F(\ol{\x}^{0,s}) - \tfrac{\a}{2}\tsum_{t=0}^{q}\|\nabla F(\ol{\x}^{t,s})\|^2
- \tfrac{\a}{4}\tsum_{t=0}^{q}\b\|\ol{\v}^{t,s}\b\|^2 \\ 
&+\a\tsum_{t=0}^{q}\b\|\ol{\v}^{t,s}-\ol{\nabla\mb{f}}(\x^{t,s})\b\|^2 
+ \tfrac{\a L^2}{n}\tsum_{t=0}^{q}\b\|\mb{x}^{t,s}-\J\mb{x}^{t,s}\b\|^2. \n
\end{align}
The proof then follows by taking the telescoping sum of~\eqref{L31} over~$s$ from~$1$ to~$S$ and taking the expectation of the resulting inequality. 

\section{Proof of Lemma~\ref{vrlem}}\label{proof_vrlem}
We first provide a useful result.
\begin{lemma}\label{tr1}
The following inequality holds:~$\forall s\geq1$,~$\forall t\in[1,q]$,~$\forall i\in\mc{V}$,~$\forall l\in[1,B]$,
\begin{align*}
\E\bigg[\Big\|\gf_{i,\tau_{i,l}^{t,s}}\big(\mb{x}_i^{t,s}\big)-\gf_{i,\tau_{i,l}^{t,s}}\big(\mb{x}_i^{t-1,s}\big)\Big\|^2\big|\F\bigg]\leq L^2 \big\|\mb{x}_i^{t,s} - \mb{x}_i^{t-1,s}\big\|^2.
\end{align*}
\end{lemma}
\begin{proof}
In the following, we denote~$\mathbbm{1}\{A\}$ as the indicator function of an event~${A\in\mc{F}}$. Observe that:~$\forall s\geq1$,~$\forall t\in[1,q]$,~$\forall i\in\mc{V}$,~$\forall l\in[1,B]$,
\begin{align}\label{tr10}
&\E\bigg[\Big\|\gf_{i,\tau_{i,l}^{t,s}}\big(\mb{x}_i^{t,s}\big)-\gf_{i,\tau_{i,l}^{t,s}}\big(\mb{x}_i^{t-1,s}\big)\Big\|^2\Big|\F\bigg] \n\\
=&~\E\!\left[\Bigg\|\sum_{j=1}^{m}\mathbbm{1}\big\{\tau_{i,l}^{t,s} = j\big\}\bigg(\gf_{i,j}\big(\mb{x}_i^{t,s}\big)-\gf_{i,j}\big(\mb{x}_i^{t-1,s}\big)\bigg)\Bigg\|^2\bigg|\F\right] \nonumber\\
=&~\sum_{j=1}^{m}\E\bigg[\mathbbm{1}\big\{\tau_{i,l}^{t,s} = j\big\}\Big\|\gf_{i,j}\big(\mb{x}_i^{t,s}\big)-\gf_{i,j}\big(\mb{x}_i^{t-1,s}\big)\Big\|^2\Big|\F\bigg] \nonumber\\
=&~\sum_{j=1}^{m}\E\Big[\mathbbm{1}\big\{\tau_{i,l}^{t,s} = j\big\}\big|\F\Big]\Big\|\gf_{i,j}\big(\mb{x}_i^{t,s}\big)-\gf_{i,j}\big(\mb{x}_i^{t-1,s}\big)\Big\|^2 \nonumber\\
=&~\frac{1}{m}\sum_{j=1}^{m}\Big\|\gf_{i,j}\big(\mb{x}_i^{t,s}\big)-\gf_{i,j}\big(\mb{x}_i^{t-1,s}\big)\Big\|^2, \n
\end{align}
where the last line uses that~$\tau_{i,l}^{t,s}$ is independent of~$\F$, i.e.,~${\E\big[\mathbbm{1}\{\tau_{i,l}^{t,s} = j\}|\F\big] =  \frac{1}{m}}$. The proof follows by using Assumption~\ref{smooth}. 
\end{proof}

\newpage
Next, we derive an upper bound on the estimation error of the average of local \texttt{SARAH} gradient estimators across the nodes at each inner loop iteration.
\begin{lemma}\label{vr00}
The following inequality holds:~$\forall s\geq1$ and~$t\in[1,q]$,
\begin{align*}
\E\Big[\big\|\ol{\v}^{t,s} - \ol{\nabla\mb{f}}(\x^{t,s}) \big\|^2\Big] \leq \frac{3\a^2L^2}{nB}\sum_{u=0}^{t-1}\E\Big[\b\|\ol{\v}^{u,s}\b\|^2\Big] + \frac{6L^2}{n^2B}\sum_{u=0}^{t}\E\Big[\big\|\x^{u,s} - \mb{J}\x^{u,s}\big\|^2\Big].
\end{align*}
\end{lemma}
\begin{proof}
For the ease of exposition, we denote:~$\forall t\in[1,q]$, $\forall s\geq1$, $\forall i\in\mc{V}$, $\forall l\in[1,B]$, 
\begin{align}\label{notation_v}
\wh{\nabla}_{i,l}^{t,s} := \gf_{i,\tau_{i,l}^{t,s}}\big(\x_{i}^{t,s}\big) - \gf_{i,\tau_{i,l}^{t,s}}\big(\x_{i}^{t-1,s}\big), 
\qquad
\wh{\nabla}_i^{t,s} := \tfrac{1}{B}\tsum_{l=1}^B\wh{\nabla}_{i,l}^{t,s}.    
\end{align}
Since~$\x_{i}^{t,s}$ and~$\x_{i}^{t-1,s}$ are~$\F$-measurable, we have
\begin{align}\label{track}
\E\Big[\wh{\nabla}_{i,l}^{t,s}|\F\Big]  
= \E\Big[\wh{\nabla}_i^{t,s}|\F\Big] 
= \nabla f_i\big(\x_i^{t,s}\big) - \nabla f_i\big(\x_i^{t-1,s}\big).
\end{align}
With the notations in~\eqref{notation_v}, the local recursive update of the gradient estimator~$\mb{v}_i^{t,s}$ described in~\cref{GT-SARAH} may be written as 
\begin{align*}
\mb{v}_i^{t,s} = \wh{\nabla}_{i}^{t,s} + \mb{v}_i^{t-1,s}, \qquad
\forall t\in[1,q],~\forall s\geq1,~\forall i\in\mc{V}.
\end{align*}
In the light of~\eqref{track}, we have the following:~$\forall s\geq1$ and~$t\in[1,q]$,
\begin{align}\label{vr1}
&\E\Big[\big\|\ol{\v}^{t,s} - \ol{\nabla\mb{f}}(\x^{t,s}) \big\|^2 \big| \F\Big]\nonumber\\
=&~\E\Bigg[\bigg\|\frac{1}{n}\sum_{i=1}^n\Big(\wh{\nabla}_i^{t,s} + \mb{v}_{i}^{t-1,s} - \nabla f_i\big(\x_i^{t,s}\big)\Big) \bigg\|^2 \bigg| \F\Bigg] \n\\
=&~\E\Bigg[\bigg\|\frac{1}{n}\sum_{i=1}^n\Big(\wh{\nabla}_i^{t,s} - \nabla f_i\big(\x_i^{t,s}\big) + \nabla f_i\big(\x_i^{t-1,s}\big) + \mb{v}_{i}^{t-1,s} - \nabla f_i\big(\x_i^{t-1,s}\big)\Big) \bigg\|^2 \bigg| \F\Bigg] \nonumber\\
=&~\E\Bigg[\bigg\|\frac{1}{n}\sum_{i=1}^n\Big(\wh{\nabla}_i^{t,s} - \nabla f_i\big(\x_i^{t,s}\big) + \nabla f_i\big(\x_i^{t-1,s}\big)\Big)\bigg\|^2 \bigg| \F\Bigg] \n\\
&+\bigg\|\frac{1}{n}\sum_{i=1}^n\Big(\mb{v}_{i}^{t-1,s} - \nabla f_i\big(\x_i^{t-1,s}\big)\Big) \bigg\|^2 \n\\
=&~\E\Bigg[\bigg\|\frac{1}{n}\sum_{i=1}^n\Big(\wh{\nabla}_i^{t,s} - \nabla f_i\big(\x_i^{t,s}\big) + \nabla f_i\big(\x_i^{t-1,s}\big)\Big)\bigg\|^2 \bigg| \F\Bigg] \n\\ 
&+\big\|\ol{\v}^{t-1,s} - \ol{\nabla\mb{f}}(\x^{t-1,s}) \big\|^2, 
\end{align}
where the third equality is due to~\eqref{track} and the fact that~$\sum_{i=1}^n\big(\mb{v}_{i}^{t-1,s} - \nabla f_i(\x_i^{t-1,s})\big)$ is~$\F$-measurable. To proceed from~\eqref{vr1}, we note that since the collection of random variables~$\big\{\tau_{i,l}^{t,s}:i\in\mc{V},l\in[1,B]\big\}$ are independent of each other and of the filtration~$\F$, by~\eqref{track}, we have:~$\forall t\in[1,q]$ and~$s\geq1$,
\begin{align}\label{Indep1}
\E\Big[\Big\langle\wh{\nabla}_i^{t,s} - \nabla f_i\big(\x_i^{t,s}\big) + \nabla f_i\big(\x_i^{t-1,s}\big), \wh{\nabla}_r^{t,s} - \nabla f_r(\x_r^{t,s}\big) + \nabla f_r\big(\x_r^{t-1,s}\big)\Big\rangle\Big|\F\Big] = 0,
\end{align}
whenever~$i,r\in\mc{V}$ such that~$i\neq r$. Similarly, we have:~$\forall t\in[1,q]$ and~$s\geq1$,~$\forall i\in\mc{V}$,
\begin{align}\label{Indep2}
\E\Big[\Big\langle\wh{\nabla}_{i,l}^{t,s} - \nabla f_i\big(\x_i^{t,s}\big) + \nabla f_i\big(\x_i^{t-1,s}\big), \wh{\nabla}_{i,h}^{t,s} - \nabla f_i(\x_i^{t,s}\big) + \nabla f_i\big(\x_i^{t-1,s}\big)\Big\rangle\Big|\F\Big] = 0,
\end{align}
whenever~$l,h\in[1,m]$ such that~$l\neq h$.
With the help of~\eqref{Indep1} and~\eqref{Indep2}, we may simplify~\eqref{vr1} in the following:~$\forall s\geq1$ and $t\in[1,q]$,
\begin{align}\label{vr3_0}
&\E\Big[\b\|\ol{\v}^{t,s} - \ol{\nabla\mb{f}}(\x^{t,s}) \b\|^2 \big| \F\Big] \n\\
=&~\frac{1}{n^2}\sum_{i=1}^n\E\bigg[\Big\|\wh{\nabla}_i^{t,s} - \nabla f_i\big(\x_i^{t,s}\big) + \nabla f_i\big(\x_i^{t-1,s}\big)\Big\|^2 \Big| \F\bigg] 
+ \big\|\ol{\v}^{t-1,s} - \ol{\nabla\mb{f}}\big(\x^{t-1,s}\big)\big\|^2 \n\\
=&~\frac{1}{n^2}\sum_{i=1}^n\E\Bigg[\bigg\|\frac{1}{B}\sum_{l=1}^B\Big(\wh{\nabla}_{i,l}^{t,s} - \nabla f_i\big(\x_i^{t,s}\big) + \nabla f_i\big(\x_i^{t-1,s}\big)\Big)\bigg\|^2 \Big| \F\Bigg] \nonumber\\
&+\left\|\ol{\v}^{t-1,s} - \ol{\nabla\mb{f}}\big(\x^{t-1,s}\big) \right\|^2 \n\\
=&~\frac{1}{(nB)^2}\sum_{i=1}^n\sum_{l=1}^B\E\bigg[\Big\|\wh{\nabla}_{i,l}^{t,s} - \nabla f_i\big(\x_i^{t,s}\big) + \nabla f_i\big(\x_i^{t-1,s}\big)\Big\|^2 \Big| \F\bigg] \\
&+\left\|\ol{\v}^{t-1,s} - \ol{\nabla\mb{f}}\big(\x^{t-1,s}\big) \right\|^2, \n
\end{align}
where the first line is due to~\eqref{Indep1} and the last line is due to~\eqref{Indep2}. To proceed from~\eqref{vr3_0}, we observe that~$\forall t\in[1,q]$, $\forall s\geq1$, $\forall i\in\mc{V}$, $\forall l\in[1,B]$,
\begin{align}\label{vrd}
\E\bigg[\Big\|\wh{\nabla}_{i,l}^{t,s} - \nabla f_i(\x_i^{t,s}) + \nabla f_i(\x_i^{t-1,s})\Big\|^2 \Big| \F\bigg]
=&~\E\bigg[\Big\|\wh{\nabla}_{i,l}^{t,s} - \E\Big[\wh{\nabla}_{i,l}^{t,s}|\F\Big]\Big\|^2 \big| \F\bigg] \n\\
\leq&~\E\bigg[\Big\|\wh{\nabla}_{i,l}^{t,s}\Big\|^2 | \F\bigg] \n\\
\leq&~L^2 \big\|\mb{x}_i^{t,s} - \mb{x}_i^{t-1,s}\big\|^2,
\end{align}
where the last line uses~\cref{tr1}.
Applying~\eqref{vrd} to~\eqref{vr3_0} yields:~$\forall s\geq1,t\in[1,q]$, 
\begin{align}\label{vr3}
\E\Big[\b\|\ol{\v}^{t,s} - \ol{\nabla\mb{f}}(\x^{t,s})\b\|^2 \b|\F\Big]
\leq \tfrac{L^2}{n^2B}\b\|\x^{t,s} - \x^{t-1,s}\b\|^2 + \big\|\ol{\v}^{t-1,s} - \ol{\nabla\mb{f}}(\x^{t-1,s}) \big\|^2.
\end{align}
We next bound the first term on the right hand side of~\eqref{vr3}. Observe that~$\forall s\geq1$ and~$t\in[1,q+1]$,
\begin{align}\label{xdiff}
\big\|\x^{t,s} - \x^{t-1,s}\big\|^2 
=&~\big\|\x^{t,s} - \mb{J}\x^{t,s} + \mb{J}\x^{t,s} - \mb{J}\x^{t-1,s} + \mb{J}\x^{t-1,s}- \x^{t-1,s}\big\|^2 \nonumber\\
\leq&~3\big\|\x^{t,s} - \mb{J}\x^{t,s}\big\|^2 + 3n\big\|\ol{\x}^{t,s}-\ol{\x}^{t-1,s}\big\|^2
 + 3\big\|\x^{t-1,s}-\mb{J}\x^{t-1,s}\big\|^2 \nonumber\\
 =&~3\big\|\x^{t,s} - \mb{J}\x^{t,s}\big\|^2 + 3n\a^2\big\|\ol{\v}^{t-1,s}\big\|^2
 + 3\big\|\x^{t-1,s}-\mb{J}\x^{t-1,s}\big\|^2.
\end{align}
Applying~\eqref{xdiff} to~\eqref{vr3} and taking the expectation of the resulting inequality leads to:~$\forall s\geq1$ and~$t\in[1,q]$,
\begin{align}\label{vr4}
\E\Big[\b\|\ol{\v}^{t,s} - \ol{\nabla\mb{f}}\big(\x^{t,s}\big) \b\|^2\Big] 
\leq&~\E\Big[\b\|\ol{\v}^{t-1,s} - \ol{\nabla\mb{f}}\big(\x^{t-1,s}\big)\b\|^2\Big] 
+ \tfrac{3\a^2L^2}{nB}\E\Big[\b\|\ol{\v}^{t-1,s}\b\|^2\Big]\nonumber\\
&+ \tfrac{3L^2}{n^2B}\E\Big[\b\|\x^{t,s} - \mb{J}\x^{t,s}\b\|^2\Big] 
+ \tfrac{3L^2}{n^2B}\E\Big[\b\|\x^{t-1,s}-\mb{J}\x^{t-1,s}\b\|^2\Big].
\end{align}
We recall the initialization of each inner loop that~$\ol{\v}^{0,s} = \ol{\nabla\mb{f}}(\x^{0,s}),\forall s\geq1$, and take the telescoping sum of~\eqref{vr4} over~$t$ from~$1$ to~$z$ to obtain:~$\forall s\geq1$ and~$\forall z\in[1,q]$,
\begin{align}\label{vr6}
\E\Big[\b\|\ol{\v}^{z,s} - \ol{\nabla\mb{f}}(\x^{z,s}) \b\|^2\Big] 
\leq&~\tfrac{3\a^2L^2}{nB}\tsum_{t=1}^{z}\E\Big[\b\|\ol{\v}^{t-1,s}\b\|^2\Big] + \tfrac{3L^2}{n^2B}\tsum_{t=1}^{z}\E\Big[\b\|\x^{t,s} - \mb{J}\x^{t,s}\b\|^2\Big] \nonumber\\
&+\tfrac{3L^2}{n^2B}\tsum_{t=1}^{z}\E\Big[\b\|\x^{t-1,s}-\mb{J}\x^{t-1,s}\b\|^2\Big].
\end{align}
The proof follows by merging the last two terms on the right hand side of~\eqref{vr6}.
\end{proof}
\begin{L4}
Summing up Lemma~\ref{vr00} over~$t$ from~$1$ to~$q$ gives:~$\forall s\geq1$,
\begin{align}\label{addon1}
\sum_{t=1}^{q}\E\Big[\b\|\ol{\v}^{t,s} - \ol{\nabla\mb{f}}(\x^{t,s})\b\|^2\Big] 
\leq&~\frac{3\a^2L^2}{nB}\sum_{t=1}^{q}\sum_{u=0}^{t-1}\E\big[\|\ol{\v}^{u,s}\|^2\big] \nonumber\\
&+ \frac{6L^2}{n^2B}\sum_{t=1}^{q}\sum_{u=0}^{t}\E\Big[\b\|\x^{u,s} - \mb{J}\x^{u,s}\b\|^2\Big]. 
\end{align}
The proof follows by relaxing the right hand side of~\eqref{addon1} on the summations and the initialization of each inner loop that~$\ol{\v}^{0,s} = \ol{\nabla\mb{f}}(\x^{0,s}),\forall s\geq1$.
\end{L4}

\section{Proof of Lemma~\ref{consensus_bound1}}\label{proof_consensus_bound1}
\subsection{Gradient tracking error}
We first provide some useful bounds on the gradient estimator tracking errors. These bounds will later be coupled with~\cref{consensus1} to formulate a dynamical system to characterize the error evolution of~\GS. The following lemma establishes an upper bound on the sum of the local gradient estimation errors across the nodes. Its proof is similar to that of Lemma~\ref{vr00}.
\begin{lemma}\label{vrs}
The following inequality holds~$\forall s\geq1$ and~$t\in[1,q]$,
\begin{align*}
\E\Big[\big\|\v^{t,s} - \nabla\mb{f}(\x^{t,s}) \big\|^2\Big] \leq \frac{3n\a^2L^2}{B}\sum_{u=0}^{t-1}\E\Big[\b\|\ol{\v}^{u,s}\b\|^2\Big] + \frac{6L^2}{B}\sum_{u=0}^{t}\E\Big[\big\|\x^{u,s} - \mb{J}\x^{u,s}\big\|^2\Big]. 
\end{align*}
\end{lemma}
\begin{proof}
See~\cref{svrs}. 
\end{proof}
We note that Lemma~\ref{vrs} does not follow directly from the results of Lemma~\ref{vrlem} because~$\v_i^{t,s}$ is \textit{not} a conditionally unbiased estimator of~$\nabla f_i(\mb{x}_i^{t,s})$ with respect to~$\mc{F}^{t,s}$.
With~\cref{vrs} at hand, we now quantify the gradient tracking errors.
\begin{lemma}\label{gt_contrc}
We have the following three statements.

(i) It holds that~$\big\|\mb{y}^{1,1}-\mb{J}\mb{y}^{1,1}\big\|^2
\leq\big\|\nabla\mb{f}\big(\x^{0,1}\big)\big\|^2$.

(ii) If~$0<\a\leq\frac{1-\lambda^2}{4\sqrt{3}L}$, the following inequality holds:~$\forall s\geq1$ and~$t\in[1,q]$,
\begin{align*}
\mathbb{E}\bigg[\frac{\|\mb{y}^{t+1,s}-\mb{J}\mb{y}^{t+1,s}\|^2}{nL^2}\bigg] \leq&~\frac{3+\lambda^2}{4}\mathbb{E}\bigg[\frac{\|\mb{y}^{t,s}-\mb{J}\mb{y}^{t,s}\|^2}{nL^2}\bigg] \nonumber\\
&+ \frac{18}{1-\lambda^2}\E\bigg[\frac{\|\x^{t-1,s}-\mb{J}\x^{t-1,s}\|^2}{n}\bigg] + \frac{6\a^2}{1-\lambda^2}\E\Big[\b\|\ol{\v}^{t-1,s}\b\|^2\Big].
\end{align*}

(iii) If~$0<\a\leq\frac{1-\lambda^2}{4\sqrt{6}L}$, the following inequality holds:~$\forall s\geq2$,
\begin{align*}
\mathbb{E}\bigg[\frac{\|\mb{y}^{1,s}-\mb{J}\mb{y}^{1,s}\|^2}{nL^2}\bigg] \leq&~\frac{3+\lambda^2}{4}\E\bigg[\frac{\|\mb{y}^{q+1,s-1}-\mb{J}\mb{y}^{q+1,s-1}\|^2}{nL^2}\bigg]   \nonumber\\
&+\frac{18}{1-\lambda^2}\E\bigg[\frac{\|\x^{q,s-1} - \mb{J}\x^{q,s-1}\|^2}{n}\bigg] + \frac{12\a^2}{1-\lambda^2}\sum_{t=0}^{q}\E\Big[\b\|\ol{\v}^{t,s-1}\b\|^2\Big] \nonumber\\
&+\frac{42}{1-\lambda^2}\sum_{t=0}^{q}\E\bigg[\frac{\|\x^{t,s-1} - \mb{J}\x^{t,s-1}\|^2}{n}\bigg]. 
\end{align*}
\end{lemma}
\begin{proof}
\emph{\textbf{(i)}} Recall that~$\v^{-1,1} = \mb{0}_{np}$,~$\y^{0,1} = \mb{0}_{np}$ and~$\v^{0,1} = \nabla\mb{f}\big(\x^{0,1}\big)$. Using the gradient tracking update at iteration~$(1,1)$ and~$\|\I_{np}-\mb{J}\| = 1$, we have:
\begin{align*}
\big\|\mb{y}^{1,1}-\mb{J}\mb{y}^{1,1}\big\|^2
=&~\big\|(\I_{np} -\mb{J})\left(\mb{W}\mb{y}^{0,1} + \mb{v}^{0,1} - \mb{v}^{-1,1}\right)\big\|^2 
\leq \big\|\nabla\mb{f}\big(\x^{0,1}\big)\big\|^2,
\end{align*}
which proves the first statement in the lemma. In the following, we prove the second and the third statements. 
We have:~$\forall s\geq1$ and~$\forall t\in[0,q]$,
\begin{align}\label{y0}
\big\|\mb{y}^{t+1,s}-\mb{J}\mb{y}^{t+1,s}\big\|^2
=&~\big\|\mb{W}\mb{y}^{t,s} + \mb{v}^{t,s} - \mb{v}^{t-1,s} -\mb{J}\left(\mb{W}\mb{y}^{t,s} + \mb{v}^{t,s} - \mb{v}^{t-1,s}\right)\big\|^2 \nonumber\\
=&~\big\|\mb{W}\mb{y}^{t,s}-\mb{J}\mb{y}^{t,s}+\left(\mb{I}_{np}-\mb{J}\right)\left(\mb{v}^{t,s} - \mb{v}^{t-1,s}\right)\big\|^2.
\end{align}
We apply the inequality that~${\|\mb{a}+\mb{b}\|^2\leq(1+\eta)\|\mb{a}\|^2+(1+\frac{1}{\eta})\|\mb{b}\|^2}$, ${\forall\mb{a},\mb{b}\in\mathbb{R}^{np}}$, with $\eta=\frac{1-\lambda^2}{2\lambda^2}$ and that~$\|\mb{I}_{np}-\mb{J}\| = 1$ to~\cref{y0} to obtain:~$\forall s\geq1$ and~$\forall t\in[0,q]$,
\begin{align}\label{yt0}
\big\|\mb{y}^{t+1,s}-\mb{J}\mb{y}^{t+1,s}\big\|^2 
\leq&~\tfrac{1+\lambda^2}{2\lambda^2}\big\|\mb{W}\mb{y}^{t,s}-\mb{J}\mb{y}^{t,s}\big\|^2 
+ \tfrac{1+\lambda^2}{1-\lambda^2}\big\|\mb{v}^{t,s} - \mb{v}^{t-1,s}\big\|^2\nonumber\\
\leq&~\tfrac{1+\lambda^2}{2}\big\|\mb{y}^{t,s}-\mb{J}\mb{y}^{t,s}\big\|^2 + \tfrac{2}{1-\lambda^2}\big\|\mb{v}^{t,s} - \mb{v}^{t-1,s}\big\|^2,
\end{align}
where the last line is due to Lemma~\ref{W}.
Next, we derive upper bounds for the last term in~\eqref{yt0} under different ranges of~$t$ and~$s$.

\textbf{\emph{(ii)}~$\forall t\in[1,q]$ and~$\forall s\geq1$.} By the update of each local~$\v_i^{t,s}$, we have that
\begin{align}\label{v_diff0}
\mathbb{E}\Big[\b\|\mb{v}^{t,s}-\mb{v}^{t-1,s}\b\|^2|\mc{F}^{t,s}\Big] 
=&\sum_{i=1}^n\mathbb{E}\Bigg[\bigg\|\frac{1}{B}\sum_{l=1}^B\Big(\gf_{i,\tau_{i,l}^{t,s}}\big(\x_{i}^{t,s}\big) - \gf_{i,\tau_{i,l}^{t,s}}\big(\x_{i}^{t-1,s}\big)\Big)\bigg\|^2\Big|\mc{F}^{t,s}\Bigg] \nonumber\\
\leq&~\frac{1}{B}\sum_{i=1}^n\sum_{l=1}^B\mathbb{E}\bigg[\Big\|\gf_{i,\tau_{i,l}^{t,s}}\big(\x_{i}^{t,s}\big) - \gf_{i,\tau_{i,l}^{t,s}}\big(\x_{i}^{t-1,s}\big)\Big\|^2\Big|\mc{F}^{t,s}\bigg] \nonumber\\
\leq&~L^2\left\|\x^{t,s} - \x^{t-1,s}\right\|^2.
\end{align}
where the last line is due to Lemma~\ref{tr1}. To proceed, we further use~\eqref{xdiff} and~\eqref{consensus2} to refine~\eqref{v_diff0} as follows:~$\forall s\geq1$ and~$\forall t\in[1,q]$,
\begin{align}\label{v_diff}
&\mathbb{E}\Big[\big\|\mb{v}^{t,s}-\mb{v}^{t-1,s}\big\|^2|\mc{F}^{t,s}\Big]\nonumber\\
\leq&~3L^2\left\|\x^{t,s} - \mb{J}\x^{t,s}\right\|^2 + 3n\a^2L^2\left\|\ol{\v}^{t-1,s}\right\|^2 + 3L^2\left\|\x^{t-1,s}-\mb{J}\x^{t-1,s}\right\|^2 \nonumber\\
\leq&~3n\a^2L^2\left\|\ol{\v}^{t-1,s}\right\|^2 + 9L^2\left\|\x^{t-1,s}-\mb{J}\x^{t-1,s}\right\|^2
+ 6\a^2L^2\left\|\y^{t,s}-\mb{J}\y^{t,s}\right\|^2.
\end{align}
We take the expectation of~\eqref{v_diff} and use it in~\eqref{yt0} to obtain:~$\forall s\geq1$ and~$\forall t\in[1,q]$,
\begin{align*}
\mathbb{E}\Big[\big\|\mb{y}^{t+1,s}-\mb{J}\mb{y}^{t+1,s}\big\|^2\Big] 
\leq&~\left(\tfrac{1+\lambda^2}{2}+\tfrac{12\a^2L^2}{1-\lambda^2}\right)\E\Big[\big\|\mb{y}^{t,s}-\mb{J}\mb{y}^{t,s}\big\|^2\Big] \nonumber\\
&+\tfrac{18L^2}{1-\lambda^2}\E\Big[\big\|\x^{t-1,s}-\mb{J}\x^{t-1,s}\big\|^2\Big]
+\tfrac{6n\a^2L^2}{1-\lambda^2}\E\Big[\big\|\ol{\v}^{t-1,s}\big\|^2\Big].
\end{align*}
The second statement in the lemma follows by the fact that~$\frac{1+\lambda^2}{2}+\frac{12\alpha^2L^2}{1-\lambda^2}\leq\frac{3+\lambda^2}{4}$ if~$0<\a\leq\frac{1-\lambda^2}{4\sqrt{3}L}$.

\textbf{\emph{(iii)}~$t = 0$ and~$\forall s\geq2$.} By the update of~\GS, we observe that:~$\forall s\geq2$,
\begin{align}\label{yt_final00}
\left\|\mb{v}^{0,s}-\mb{v}^{-1,s}\right\|^2          
=&~\left\|\nabla\mb{f}(\x^{q+1,s-1}) - \mb{v}^{q,s-1}\right\|^2
\n\\   
=&~\left\|\nabla\mb{f}(\x^{q+1,s-1}) - \nabla\mb{f}(\x^{q,s-1}) + \nabla\mb{f}(\x^{q,s-1}) - \mb{v}^{q,s-1}\right\|^2 \nonumber\\
\leq&~2L^2\left\|\x^{q+1,s-1} - \x^{q,s-1}\right\|^2 + 2\left\|\nabla\mb{f}\big(\x^{q,s-1}\big) - \v^{q,s-1} \right\|^2,
\n\\
\leq&~6L^2\left\|\x^{q+1,s-1} - \J\x^{q+1,s-1}\right\|^2
+ 6n\a^2L^2\left\|\ol{\v}^{q,s-1}\right\|^2
\n\\
&+ 6L^2\left\|\x^{q,s-1} - \J\x^{q,s-1}\right\|^2 
+ 2\left\|\nabla\mb{f}\big(\x^{q,s-1}\big) - \v^{q,s-1} \right\|^2           \n\\
\leq&~18L^2\left\|\x^{q,s-1} - \J\x^{q,s-1}\right\|^2
+ 6n\a^2L^2\left\|\ol{\v}^{q,s-1}\right\|^2
\n\\
&+ 12\a^2L^2\left\|\y^{q+1,s-1} - \J\y^{q+1,s-1}\right\|^2 
+ 2\left\|\nabla\mb{f}\big(\x^{q,s-1}\big) - \v^{q,s-1} \right\|^2,
\end{align}
where the first inequality uses the~$L$-smoothness of each~$f_i$, the second inequality uses~\eqref{xdiff}, and the last inequality uses~\eqref{consensus2}. Taking the expectation of~\eqref{yt_final00} and then using Lemma~\ref{vrs} gives:~$\forall s\geq2$,
\begin{align}\label{yt_final0}
\mathbb{E}\Big[\b\|\mb{v}^{0,s}-\mb{v}^{-1,s}\b\|^2\Big]
\leq&~18L^2\E\Big[\b\|\x^{q,s-1}-\mb{J}\x^{q,s-1}\b\|^2\Big] +6n\a^2L^2\tsum_{t=0}^{q}\E\Big[\b\|\ol{\v}^{t,s-1}\b\|^2\Big]
\nonumber\\
&+ 12\a^2L^2\E\Big[\b\|\y^{q+1,s-1}-\mb{J}\y^{q+1,s-1}\b\|^2\Big] \nonumber\\
&+ \tfrac{12L^2}{B}\tsum_{t=0}^{q}\E\Big[\b\|\x^{t,s-1} - \mb{J}\x^{t,s-1}\b\|^2\Big]. 
\end{align}
We recall from~\cref{yt0} that~$\forall s\geq2$,
\begin{align}\label{yt_final}
\left\|\mb{y}^{1,s}-\mb{J}\mb{y}^{1,s}\right\|^2 
\leq\tfrac{1+\lambda^2}{2}\left\|\mb{y}^{q+1,s-1}-\mb{J}\mb{y}^{q+1,s-1}\right\|^2 + \tfrac{2}{1-\lambda^2}\left\|\mb{v}^{0,s} - \mb{v}^{-1,s}\right\|^2.
\end{align}
We finally apply~\cref{yt_final0} to~\cref{yt_final} to obtain:~$\forall s\geq2$,
\begin{align*}
\mathbb{E}\Big[\big\|\mb{y}^{1,s}-\mb{J}\mb{y}^{1,s}\b\|^2\Big] 
\leq&~\left(\tfrac{1+\lambda^2}{2}+\tfrac{24\a^2L^2}{1-\lambda^2}\right)\E\Big[\b\|\mb{y}^{q+1,s-1}-\mb{J}\mb{y}^{q+1,s-1}\b\|^2\Big]
\nonumber\\
&+\tfrac{36L^2}{1-\lambda^2}\E\Big[\b\|\x^{q,s-1}-\mb{J}\x^{q,s-1}\b\|^2\Big]
+ \tfrac{12n\a^2L^2}{1-\lambda^2}\tsum_{t=0}^{q}\E\Big[\b\|\ol{\v}^{t,s-1}\b\|^2\Big] \nonumber\\
&+ \tfrac{24L^2}{B(1-\lambda^2)}\tsum_{t=0}^{q}\E\Big[\b\|\x^{t,s-1} - \mb{J}\x^{t,s-1}\b\|^2\Big]. 
\end{align*}
We note that~$\frac{1+\lambda^2}{2}+\frac{24\a^2L^2}{1-\lambda^2}\leq\frac{3+\lambda^2}{4}$ if~$0<\a\leq\frac{1-\lambda^2}{4\sqrt{6}L}$ and then the third statement in the lemma follows.
\end{proof}

\subsection{\GS~as a linear time-invariant (LTI) system}
With the help of \cref{consensus1} and~\cref{gt_contrc}, we now abstract~\GS~with an LTI system to quantify jointly the state consensus and the gradient tracking error. 
\begin{lemma}\label{GS_LTI}
If the step-size~$\a$ follows that~$0<\alpha\leq\frac{1-\lambda^2}{4\sqrt{6}L}$, then we have
\begin{align}
&\mb{u}^{t,s} \leq \G\mb{u}^{t-1,s} + \mb{b}^{t-1,s}, \qquad\qquad\qquad\qquad\qquad\qquad\forall s\in[1,S]~\text{and}~t\in[1,q], \label{D1}\\
&\mb{u}^{0,s} \leq \G\mb{u}^{q,s-1} + \mb{b}^{q,s-1} + \sum_{t=0}^{q}\Big(\mb{b}^{t,s-1}
+ \mb{H}\mb{u}^{t,s-1}\Big), \qquad\qquad~\forall s\in[2,S], \label{D2}
\end{align}
where,~$\forall s\geq1$ and~$\forall t\in[0,q]$,
\begin{align*}
\mb{u}^{t,s}&:=
\begin{bmatrix}
\dfrac{1}{n}\mbb{E}\Big[\big\|\mb{x}^{t,s}-\mb{J}\mb{x}^{t,s}\big\|^2\Big]\\[1.5ex]
\dfrac{1}{nL^2}\mbb{E}\Big[\big\|\mb{y}^{t+1,s}-\mb{J}\mb{y}^{t+1,s}\big\|^2\Big]
\end{bmatrix},
&\mb{b}&:=
\begin{bmatrix}
0\\[1.5ex]
\dfrac{12\alpha^2}{1-\lambda^2}
\end{bmatrix},\quad\mb{b}^{t,s}:=
\mb b \mathbb{E}\Big[\b\|\ol{\v}^{t,s}\b\|^2\Big],\\
\mb{G}&:=
\begin{bmatrix}
\dfrac{1+\lambda^2}{2}  & \dfrac{2\alpha^2L^2}{1-\lambda^2} \\[2ex]
\dfrac{18}{1-\lambda^2}  & \dfrac{3+\lambda^2}{4}
\end{bmatrix},
&\mb{H}&:= 
\begin{bmatrix}
0 & 0\\[1ex]
\dfrac{42}{1-\lambda^2} & 0
\end{bmatrix}.
\end{align*}
\end{lemma}
\begin{proof}
Write the inequalities in~\cref{consensus1} and~\cref{gt_contrc}
jointly in a matrix form.
\end{proof}

We next derive the range of the step-size~$\alpha$ such that~$\rho(\mb G)<1$, i.e. the LTI system does not diverge, with the help of the following lemma. 
\begin{lemma}[\!\!\cite{matrix_analysis}]\label{rho_bound}Let~$\mb{X}\in\mathbb{R}^{d\times d}$ be (entry-wise) non-negative and~$\mb{x}\in\mathbb{R}^d$ be (entry-wise) positive. If~$\mb{X}\mb{x}<\mb{x}$ (entry-wise), then~$\rho(\mb{X})<1$.
\end{lemma}

\begin{lemma}\label{st}
If the step-size~$\alpha$ follows that~$0<\a<\frac{(1-\lambda^2)^2}{8\sqrt{5}L}$, then~$\rho(\mb G)<1$ and therefore~$\sum_{k=0}^{\infty}\mb{G}^k$ is convergent such that $\sum_{k=0}^{\infty}\mb{G}^k = (\mb{I}_2-\mb{G})^{-1}$.
\end{lemma}
\begin{proof}
In the light of~Lemma~\ref{rho_bound}, we solve the range of~$\alpha$ and a positive vector~$\bds\ve = [\ve_1,\ve_2]^\top$ such that~$\mb{G}\bds\ve < \bds\ve$, which is equivalent to the following two inequalities.
\begin{align}\label{e}
\left\{
\begin{array}{ll}
\tfrac{1+\lambda^2}{2}\ve_1 + \tfrac{2\alpha^2L^2}{1-\lambda^2}\ve_2<\ve_1\\[1.8ex]
\tfrac{18}{1-\lambda^2}\ve_1 + \tfrac{3+\lambda^2}{4}\ve_2 < \ve_2
\end{array}
\right.
\iff
\left\{
\begin{array}{ll}
\alpha^2<\tfrac{(1-\lambda^2)^2}{4L^2}\tfrac{\ve_1}{\ve_2}\\[1.8ex]
\tfrac{\ve_1}{\ve_2}  < \tfrac{(1-\lambda^2)^2}{72}
\end{array}
\right.
\end{align}
According to the second inequality of~\eqref{e}, we set~$\ve_1/\ve_2 = (1-\lambda^2)^2/80$ and the proof follows by using it in the first inequality of~\eqref{e} to solve for the range of~$\a$.
\end{proof}

Based on Lemma~\ref{st}, the LTI system is stable under an appropriate step-size~$\alpha$ and therefore we can solve the LTI system to obtain the following lemma, the proof of which is deferred to~\cref{proof_F} for the ease of exposition.
\begin{lemma}\label{F}
If~$0<\alpha<\frac{(1-\lambda^2)^2}{8\sqrt{5}L}$, then the following inequality holds.
\begin{align*}
\Big(\I_2-(\I_2-\G)^{-1}\mb{H}\Big)\sum_{s=1}^{S}\sum_{t=0}^{q}\mb{u}^{t,s} 
\leq&~
(\I_2-\G)^{-1}\mb{u}^{0,1} + 2(\I_2-\G)^{-1}\sum_{s=1}^{S}\sum_{t=0}^{q}\mb{b}^{t,s}.   
\end{align*}
\end{lemma}
\begin{proof}
See~\cref{proof_F}.
\end{proof}

In the following lemma, we compute~$(\I_2-\G)^{-1}$ and~$(\I_2-\G)^{-1}\mb{b}$.
\begin{lemma}\label{I-G}
If~$0<\alpha\leq\frac{(1-\lambda^2)^2}{24L}$, then the following entry-wise inequality holds:
\begin{align*}
(\I_2-\G)^{-1} \leq
\begin{bmatrix}
\frac{4}{1-\lambda^2} & \frac{32\a^2L^2}{(1-\lambda^2)^3} \\[1.5ex]
\frac{288}{(1-\lambda^2)^3} & \frac{8}{1-\lambda^2}
\end{bmatrix},
\quad
(\I_2-\G)^{-1}\mb{b}\leq
\begin{bmatrix}
\frac{384\a^4L^2}{(1-\lambda^2)^4} \\[1.5ex]
\frac{96\alpha^2}{(1-\lambda^2)^2}
\end{bmatrix}.
\end{align*}
\end{lemma}
\begin{proof}
We first derive a lower bound for~$\det(\I_2-\G)$. Note that if~$0<\alpha\leq\tfrac{(1-\lambda^2)^2}{24L}$, then~$\det(\I_2-\G) = \frac{(1-\lambda^2)^2}{8} - \frac{36\a^2L^2}{(1-\lambda^2)^2} \geq \frac{(1-\lambda^2)^2}{16}$ and therefore
\begin{align*}
(\I_2-\G)^{-1} \leq \tfrac{16}{(1-\lambda^2)^2}
\begin{bmatrix}
\tfrac{1-\lambda^2}{4} & \tfrac{2\a^2L^2}{1-\lambda^2} 
\\[1.5ex]
\tfrac{18}{1-\lambda^2} & \tfrac{1-\lambda^2}{2}
\end{bmatrix}
=
\begin{bmatrix}
\tfrac{4}{1-\lambda^2} & \tfrac{32\a^2L^2}{(1-\lambda^2)^3} \\[1.5ex]
\tfrac{288}{(1-\lambda^2)^3} & \tfrac{8}{1-\lambda^2}
\end{bmatrix},
\end{align*}
and the proof follows by the definition of~$\mb{b}$ in Lemma~\ref{GS_LTI}.
\end{proof}
\subsection{Proof of Lemma~\ref{consensus_bound1}}
Using Lemma~\ref{I-G}, we have: if~$0<\a\leq\frac{(1-\lambda^2)^2}{8\sqrt{42}L}$,
\begin{align}\label{G1}
\I_2 - (\I_2-\G)^{-1}\mb{H} 
\geq
\begin{bmatrix}
1-\tfrac{1344\alpha^2 L^2}{(1-\lambda^2)^4} & 0 \\[1.5ex]
-\tfrac{336}{(1-\lambda^2)^2} & 1
\end{bmatrix}
\geq
\begin{bmatrix}
\tfrac{1}{2} & 0 \\[1.5ex]
-\tfrac{336}{(1-\lambda^2)^2} & 1
\end{bmatrix}.
\end{align}
Finally, we apply~\eqref{G1} and Lemma~\ref{I-G} to Lemma~\ref{F} to obtain
\begin{align}\label{addon3}
\frac{1}{2n}\sum_{s=1}^S\sum_{t=0}^q\mbb{E}\Big[\big\|\mb{x}^{t,s}-\mb{J}\mb{x}^{t,s}\big\|^2\Big] 
\leq&~\frac{32\a^2}{n(1-\lambda^2)^3}\E\Big[\big\|\mb{y}^{1,1}-\mb{J}\mb{y}^{1,1}\big\|^2\Big] \nonumber\\
&+ \frac{768\a^4L^2}{(1-\lambda^2)^4}\sum_{s=1}^S\sum_{t=0}^q\mathbb{E}\Big[\big\|\ol{\v}^{t,s}\big\|^2\Big].
\end{align}
The proof follows by applying the first statement in Lemma~\ref{gt_contrc} to~\eqref{addon3}.

\newpage
\section{Proof of Lemma~\ref{DS_sum2}}\label{proof_DS_sum2}
We have:~$\forall s\geq1$ and $\forall t\in[0,q]$,
\begin{align}\label{i00}
\tfrac{1}{2n}\tsum_{i=1}^n\E\Big[\b\|\nabla F\big(\x_i^{t,s}\big)\b\|^2\Big]  
\leq&~\tfrac{1}{n}\tsum_{i=1}^n\E\Big[\b\|\nabla F\big(\x_i^{t,s}\big)-\nabla F(\ol{\x}^{t,s})\b\|^2\Big]
+\E\Big[\b\|\nabla F(\ol{\x}^{t,s})\b\|^2\Big] \nonumber\\
\leq&~\tfrac{L^2}{n}\E\Big[\b\|\x^{t,s} - \J\x^{t,s}\b\|^2\Big]
+\E\Big[\b\|\nabla F(\ol{\x}^{t,s})\b\|^2\Big],
\end{align}
where the second line is due to the $L$-smoothness of $F$.
Since $F$ is bounded below by~$F^*$, we may apply~\eqref{i00} to Lemma~\ref{DS_sum1} to obtain the following: if~$0<\a\leq\frac{1}{2L}$,
\begin{align}\label{i10}
F^* \leq&~ F\big(\ol{\x}^{0,1}\big) - \tfrac{\a}{4n}\tsum_{i=1}^n\tsum_{s=1}^S\tsum_{t=0}^{q}\E\left[\left\|\nabla F\big(\x_i^{t,s}\big)\right\|^2\right]
- \tfrac{\a}{4}\tsum_{s=1}^S\tsum_{t=0}^{q}\E\left[\left\|\ol{\v}^{t,s}\right\|^2\right] \nonumber\\
&+ \a\tsum_{s=1}^S\tsum_{t=0}^{q}\E\left[\left\|\ol{\v}^{t,s}-\ol{\nabla\mb{f}}(\x^{t,s})\right\|^2\right] + \frac{3\a L^2}{2}\tsum_{s=1}^S\tsum_{t=0}^{q}{\E\Big[\frac{\|\mb{x}^{t,s}-\mb{J}\mb{x}^{t,s}\|^2}{n}\Big]}.
\end{align}
We then apply Lemma~\ref{vrlem} to~\eqref{i10} to obtain: if~$0<\a\leq\frac{1}{2L}$,
\begin{align}\label{ifi}
F^*
\leq&~F\big(\ol{\x}^{0,1}\big) - \tfrac{\a}{4n}\tsum_{i=1}^n\tsum_{s=1}^S\tsum_{t=0}^{q}\E\left[\left\|\nabla F\big(\x_i^{t,s}\big)\right\|^2\right]
- \tfrac{\a}{8}\tsum_{s=1}^S\tsum_{t=0}^{q}\E\left[\left\|\ol{\v}^{t,s}\right\|^2\right] \nonumber\\
&+\a L^2\left(\tfrac{3}{2}+\tfrac{6q}{nB}\right)\tsum_{s=1}^S\tsum_{t=0}^{q}\E\Big[\frac{\|\mb{x}^{t,s}-\mb{J}\mb{x}^{t,s}\|^2}{n}\Big]  \nonumber\\
&- \tfrac{\a}{8}\left(1-\tfrac{24q\a^2L^2}{nB}\right)\tsum_{s=1}^S\tsum_{t=0}^{q}\E\left[\left\|\ol{\v}^{t,s}\right\|^2\right].
\end{align}
If~$0<\a\leq\frac{\sqrt{nB}}{2\sqrt{6q}L}$ then~$1-\frac{24\a^2qL^2}{nB}\geq0$ and thus the last term in~\eqref{ifi} may be dropped. 
We finally apply Lemma~\ref{consensus_bound1} to~\eqref{ifi} to obtain: if~$0<\alpha\leq\min\Big\{\frac{(1-\lambda^2)^2}{4\sqrt{42}},\sqrt{\frac{nB}{6q}}\Big\}\frac{1}{2L}$, 
\begin{align*}
F^*
\leq&~F\big(\ol{\x}^{0,1}\big) 
- \tfrac{\a}{4n}\tsum_{i=1}^n\tsum_{s=1}^S\tsum_{t=0}^{q}\E\Big[\big\|\nabla F\big(\x_i^{t,s}\big)\big\|^2\Big]
+\left(\tfrac{7}{4}+\tfrac{6q}{nB}\right) \tfrac{64\a^3L^2}{(1-\lambda^2)^3}\tfrac{\|\nabla\mb{f}(\mb{x}^{0,1})\|^2}{n}\nonumber\\
&-\tfrac{\a L^2}{4}\tsum_{s=1}^S\tsum_{t=0}^{q}\E\Big[\frac{\|\mb{x}^{t,s}-\mb{J}\mb{x}^{t,s}\|^2}{n}\Big] \n\\
&-\tfrac{\a}{8}\left(1 - \left(\tfrac{7}{4}+\tfrac{6q}{nB}\right)\tfrac{12288\a^4L^4}{(1-\lambda^2)^4}\right)\tsum_{s=1}^S\tsum_{t=0}^q\mathbb{E}\Big[\b\|\ol{\v}^{t,s}\b\|^2\Big].
\end{align*}
We observe that if~$0<\a\leq\big(\frac{4nB}{7nB+24q}\big)^{\sfrac{1}{4}}\frac{1-\lambda^2}{12L}$, then~$1 - (\frac{7}{4}+\frac{6q}{nB})\frac{12288\a^4L^4}{(1-\lambda^2)^4}\geq0$ and thus the last term in the above inequality may be dropped; the proof follows.

\section{Proof of Lemma~\ref{vrs}}\label{svrs}
In the following, we use the notation in~\eqref{notation_v}. Using the update of each local recursive gradient estimator~$\v_i^{t,s}$,
we have that:~$\forall i\in\mc{V},\forall s\geq1$ and~$t\in[1,q]$,
\begin{align}
&\E\Big[\b\|\v_i^{t,s} - \nabla f_i(\x_i^{t,s}) \b\|^2 \big| \F\Big]    \nonumber\\
=&~\E\bigg[\Big\|\wh{\nabla}_{i}^{t,s} - \nabla f_i\big(\x_i^{t,s}\big) + \nabla f_i\big(\x_i^{t-1,s}\big) + \mb{v}_{i}^{t-1,s} - \nabla f_i\big(\x_i^{t-1,s}\big) \Big\|^2 \Big| \F\bigg]  \nonumber\\
=&~\E\bigg[\Big\|\tfrac{1}{B}\tsum_{l=1}^B\Big(\wh{\nabla}_{i,l}^{t,s} - \nabla f_i\big(\x_i^{t,s}\big) + \nabla f_i\big(\x_i^{t-1,s}\big)\Big)\Big\|^2 \Big|\F\bigg] 
+\big\|\mb{v}_{i}^{t-1,s} - \nabla f_i\big(\x_i^{t-1,s}\big)\big\|^2, \nonumber\\
=&~\tfrac{1}{B^2}\tsum_{l=1}^B\E\bigg[\Big\|\wh{\nabla}_{i,l}^{t,s} - \nabla f_i\big(\x_i^{t,s}\big) + \nabla f_i\big(\x_i^{t-1,s}\big)\Big\|^2 \Big|\F\bigg]
+\big\|\mb{v}_{i}^{t-1,s} - \nabla f_i\big(\x_i^{t-1,s}\big)\big\|^2, \nonumber\\
\leq&~\tfrac{1}{B^2}\tsum_{l=1}^B\E\bigg[\Big\|\gf_{i,\t_{i,l}^{t,s}}\big(\x_{i}^{t,s}\big) - \gf_{i,\t_{i,l}^{t,s}}\big(\x_{i}^{t-1,s}\big)\Big\|^2 \Big|\F\bigg] + \big\|\mb{v}_{i}^{t-1,s} - \nabla f_i\big(\x_i^{t-1,s}\big)\big\|^2,
\nonumber\\
\leq&~\tfrac{L^2}{B}\big\|\x_{i}^{t,s} - \x_{i}^{t-1,s}\big\|^2 + \b\|\mb{v}_{i}^{t-1,s} - \nabla f_i\big(\x_i^{t-1,s}\big)\b\|^2.
\n
\end{align}
The above derivations follow a similar line of arguments as in the proof of Lemma~\ref{vr00} and hence we omit the details here. Summing up the last inequality above over~$i$ from~$1$ to~$n$ and taking the expectation, we have:~$\forall s\geq1$ and~$t\in[1,q]$, 
\begin{align}\label{1}
\E\Big[\b\|\v^{t,s} - \nabla\mb{f}\big(\x^{t,s}\big) \b\|^2 \Big]
\leq \E\Big[\tfrac{L^2}{B}\b\|\x^{t,s} - \x^{t-1,s}\b\|^2 + \b\|\mb{v}^{t-1,s} - \nabla\mb{f}\big(\x^{t-1,s}\big)\b\|^2\Big].
\end{align}
Recall from~\eqref{xdiff} that~$\forall s\geq1$ and~$t\in[1,q]$,
\begin{align}\label{2}
\big\|\x^{t,s} - \x^{t-1,s}\big\|^2 
\leq 3\big\|\x^{t,s} - \mb{J}\x^{t,s}\big\|^2 + 3n\a^2\big\|\ol{\v}^{t-1,s}\big\|^2
 + 3\big\|\x^{t-1,s} - \mb{J}\x^{t-1,s}\big\|^2.
\end{align}
Applying~\eqref{2} to~\eqref{1} obtains:~$\forall s\geq1$ and~$t\in[1,q]$,
\begin{align}\label{3}
\E\Big[\b\|\v^{t,s} - \nabla\mb{f}(\x^{t,s})\b\|^2\Big]
\leq&~\E\Big[\b\|\v^{t-1,s} - \nabla\mb{f}(\x^{t-1,s})\b\|^2\Big] + \tfrac{3n\a^2L^2}{B}\E\Big[\b\|\ol{\v}^{t-1,s}\b\|^2\Big] \n\\
&+ \tfrac{3L^2}{B}\E\Big[\b\|\x^{t,s} - \mb{J}\x^{t,s}\b\|^2\Big]
+ \tfrac{3L^2}{B}\E\Big[\b\|\x^{t-1,s}-\mb{J}\x^{t-1,s}\b\|^2\Big]. 
\end{align}
Recall that~$\v^{0,s} = \nabla\mb{f}(\x^{0,s}),\forall s\geq1$, and we take the telescoping sum of~\cref{3} over~$t$ to obtain:~$\forall s\geq1$ and~$t\in[1,q]$,
\begin{align*}
\E\Big[\b\|\v^{t,s} - \nabla\mb{f}(\x^{t,s}) \b\|^2\Big] 
\leq&~\tfrac{3n\a^2L^2}{B}\tsum_{u=1}^{t}\E\Big[\b\|\ol{\v}^{u-1,s}\b\|^2\Big] + \tfrac{3L^2}{B}\tsum_{u=1}^{t}\E\Big[\b\|\x^{u,s} - \J\x^{u,s}\b\|^2\Big] \nonumber\\
&+ \tfrac{3L^2}{B}\tsum_{u=1}^{t}\E\Big[\b\|\x^{u-1,s}-\J\x^{u-1,s}\b\|^2\Big].
\end{align*}
The proof follows by merging the last two terms on the RHS of the inequality above.

\section{Proof of Lemma~\ref{F}}\label{proof_F}
\subsection{Step 1: A loop-less dynamical system}\label{loopless_sys}
For the ease of calculations, we first write the LTI system in Lemma~\ref{GS_LTI} in a equivalent \textit{loopless} form. To do this, we unroll the original \textit{double loop} sequences~$\{\mb{u}^{t,s}\}$ and~$\{\mb{b}^{t,s}\}$, where~$t\in[0,q]$ and~$s\in[1,S]$, respectively as \emph{loopless} sequences~$\{\mb{u}^k\}$ and~$\{\mb{b}^k\}$, where~$k\in[0,S(q+1)-1]$, as follows:
\begin{align}\label{2_1}
\mb{u}^{k} := \mb{u}^{t,s},
\quad \mb{b}^{k}:=\mb{b}^{t,s}, \qquad\mbox{where}~k = t + (s-1)(q+1),
\end{align}
for~$t\in[0,q]$ and $s\in[1,S]$.
Reversely, given~$\mb{u}^k$ and~$\mb{b}^{k}$, for~$k\in[0,S(q+1)-1]$, we can find their positions in the original double loop sequence, $\mb{u}^{t,s}$ and~$\mb{b}^{t,s}$, by
\begin{align}\label{1_2}
t = \mod(k,q+1)~\text{and}~s 
= \floor*{k/(q+1)} + 1,
\qquad \text{for}~k\in[0,S(q+1)-1].
\end{align}  
This one-on-one correspondence is visualized in~\cref{D2S}.

\begin{table}[!ht]\label{D2S}
\footnotesize
\caption{The one-on-one mapping between the single-loop sequences~$\{\mb{u}^{k}\},\{\mb{b}^{k}\}$ for~$k\in[0,S(q+1)-1]$ and the double-loop sequences~$\{\mb{u}^{t,s}\},\{\mb{b}^{t,s}\}$ for~$s\in[1,S]$ and~$t\in[0,q]$.}
\begin{center}
\begin{tabular}{|c|c|}
\hline
$ k$ & $ (t,s) $ \\ \hline
$0, \cdots,q$ & $(0,1),\cdots,(q,1)$  \\ 
$q+1, \cdots,2q+1$ & $(0,2),\cdots,(q,2)$ \\ 
$\cdots$ & $\cdots$        \\ 
$(S-1)(q+1),\cdots,S(q+1)-1$ & $(0,S),\cdots,(q,S)$ \\ \hline
\end{tabular}
\end{center}
\end{table}

With~\eqref{2_1} and~\eqref{1_2} at hand, it can be verified that the following single-loop system is equivalent to the double loop system in~\eqref{D1}~and~\eqref{D2}. 
For~${k\in[1,S(q+1)-1]}$,
\begin{align}
&\mb{u}^{k} \leq \G\mb{u}^{k-1} + \mb{b}^{k-1}, \qquad\qquad\qquad\qquad\qquad\qquad\qquad~\text{if}~\mod(k,q+1) \neq 0.
\label{r1}\\
&\mb{u}^{z(q+1)} \leq \G\mb{u}^{z(q+1)-1} + \mb{b}^{z(q+1)-1} 
+ \tsum_{r=(z-1)(q+1)}^{z(q+1)-1} \mb{h}^r,~~~~~\forall z\in[1,S-1], \label{r2}
\end{align}
where~$\mb{h}^k := \mb{b}^k + \mb{H}\mb{u}^k,\forall k\in[0,S(q+1)-1]$. The system in~\eqref{r1} and~\eqref{r2} can be further written equivalently as the following:~$\forall k\in[1,S(q+1)-1]$,
\begin{align}\label{recursion_merged}
\mb{u}^{k} 
\leq&~\G\mb{u}^{k-1} + \mb{d}^k
\end{align}
where~$\mb{d}^k := 
\mb{b}^{k-1}
+ \mathbbm{1}\big\{\mod (k,q+1) = 0\big\}\sum_{r=k-(q+1)}^{k-1} \mb{h}^r$, $\mathbbm{1}\{\cdot\}$ is the indicator function of an event,
and~$\sum_{r = k-(q+1)}^{k-1}\mb{h}^r := 0$ for~$k\in[1,q]$.

\subsection{Step 2: Analyzing the recursion}
We recursively apply~\eqref{recursion_merged} over~$k$ to obtain:~$\forall k\in[1,S(q+1)-1]$, 
\begin{align}\label{addon4}
\mb{u}^{k} 
\leq \G^{k}\mb{u}^{0} + \tsum_{r=1}^{k}\G^{k-r}\mb{d}^{r}.
\end{align}
Summing up~\eqref{addon4} over~$k$ from~$0$ to~$S(q+1)-1$ gives: if~$0<\a\leq\frac{(1-\lambda^2)^2}{8\sqrt{5}L}$,
\begin{align}\label{S1}
\tsum_{k=0}^{S(q+1)-1}\mb{u}^{k} 
\leq&~
\tsum_{k=0}^{S(q+1)-1}\G^k\mb{u}^0 + \tsum_{k=1}^{S(q+1)-1}\tsum_{r=1}^{k}\G^{k-r}\mb{d}^{r} \n\\
\leq&~
\left(\tsum_{k=0}^{\infty}\G^k\right)\mb{u}^0 + \tsum_{k=1}^{S(q+1)-1}\left(\sum_{r=1}^{\infty}\G^{r}\right)\mb{d}^{k} \n\\
=&~\left(\I_2 - \G\right)^{-1}\mb{u}^0 + \left(\I_2 - \G\right)^{-1}\tsum_{k=1}^{S(q+1)-1}\mb{d}^{k}.
\end{align}
To proceed, we recall the definition of~$\mb{d}^k$ and~$\mb{h}^k$ in~\cref{loopless_sys} and observe that
\begin{align}\label{S2}
\tsum_{k=1}^{S(q+1)-1}\mb{d}^{k}
=&~\tsum_{k=0}^{S(q+1)-2}\mb{b}^{k}
+ \tsum_{k=1}^{S(q+1)-1}\Big(\mathbbm{1}\big\{\mod (k,q+1) = 0\big\}\tsum_{r=k-(q+1)}^{k-1} \mb{h}^r\Big) \n\\
=&~\tsum_{k=0}^{S(q+1)-2}\mb{b}^{k}
+ \tsum_{z=1}^{S-1}\left(\tsum_{r=(z-1)(q+1)}^{z(q+1)-1} \mb{h}^r\right) \n\\
=&~\tsum_{k=0}^{S(q+1)-2}\mb{b}^{k}
+ \tsum_{k=0}^{(S-1)(q+1)-1}\mb{h}^k \n\\
\leq&~2\tsum_{k=0}^{S(q+1)-1}\mb{b}^{k}
+ \tsum_{k=0}^{(S-1)(q+1)-1}\mb{H}\mb{u}^k,
\end{align}
where the first line and the last line are due to the definition of~$\mb{d}^k$ and~$\mb{h}^k$ respectively. 
Finally, we use~\eqref{S2} in~\eqref{S1} to obtain: if~$0<\a\leq\frac{(1-\lambda^2)^2}{8\sqrt{5}L}$, then
\begin{align*}
\tsum_{k=0}^{S(q+1)-1}\mb{u}^{k} 
\leq&~
(\I_2-\G)^{-1}\mb{u}^0 + 2(\I_2-\G)^{-1}\tsum_{k=0}^{S(q+1)-1}\mb{b}^k \nonumber\\ 
&+ (\I_2-\G)^{-1}\mb{H}\tsum_{k=0}^{S(q+1)-1}\mb{u}^k,
\end{align*}
which is the same as
\begin{align*}
\left(\I_2 - (\I_2-\G)^{-1}\mb{H}\right)\tsum_{k=0}^{S(q+1)-1}\mb{u}^{k} 
\leq
(\I_2-\G)^{-1}\mb{u}^0 + 2(\I_2-\G)^{-1}\tsum_{k=0}^{S(q+1)-1}\mb{b}^k. 
\end{align*}
We conclude the proof of Lemma~\ref{F} by rewriting the above inequality in the original double loop form. 

\bibliographystyle{siamplain}
\bibliography{reference}
\end{document}